\documentclass[a4paper,11pt]{article}

\usepackage[T1]{fontenc}
\usepackage{lmodern}

\usepackage{dsfont}

\usepackage[latin1]{inputenc}
\usepackage{amsmath}
\usepackage{amsthm}
\usepackage{amssymb}
\usepackage{mathrsfs}
\usepackage{graphicx}
\usepackage[all]{xy}
\usepackage{hyperref}

\usepackage{stmaryrd}
\usepackage{caption}

\usepackage{abstract} 

\newtheorem{thm}{Theorem}[section]
\newtheorem{cor}[thm]{Corollary}
\newtheorem{claim}[thm]{Claim}
\newtheorem{fact}[thm]{Fact}

\newtheorem{lemma}[thm]{Lemma}
\newtheorem{prop}[thm]{Proposition}

\theoremstyle{definition}
\newtheorem{definition}[thm]{Definition}
\newtheorem{ex}[thm]{Example}
\newtheorem{remark}[thm]{Remark}
\newtheorem{question}[thm]{Question}

\title{Morphisms between right-angled Coxeter groups and the embedding problem in dimension two}
\date{\today}
\author{Anthony Genevois}

\begin{document}

\maketitle

\begin{abstract}
In this article, given two finite simplicial graphs $\Gamma_1$ and $\Gamma_2$, we state and prove a complete description of the possible morphisms $C(\Gamma_1) \to C(\Gamma_2)$ between the right-angled Coxeter groups $C(\Gamma_1)$ and $C(\Gamma_2)$. As an application, assuming that $\Gamma_2$ is triangle-free, we show that, if $C(\Gamma_1)$ is isomorphic to a subgroup of $C(\Gamma_2)$, then the ball of radius $8|\Gamma_1||\Gamma_2|$ in $C(\Gamma_2)$ contains the basis of a subgroup isomorphic to $C(\Gamma_1)$. This provides an algorithm determining whether or not, among two given two-dimensional right-angled Coxeter groups, one is isomorphic to a subgroup of the other. 
\end{abstract}

\tableofcontents

\section{Introduction}

\noindent
Given two groups, a basic question one may ask is whether one of them is isomorphic to a subgroup of the other. In full generality, this question turns out to be extremely difficult, so usually one focuses on a specific class of groups. For instance, significant work has been made about the embedding problem among right-angled Artin groups \cite{Cocontraction, MR3039768, MR3436157, MR3447104, MR3692568}. In this article, inspired by the recent work \cite{RACGStallings}, we focus on right-angled Coxeter groups.  

\medskip \noindent
Recall that, given a simplicial graph $\Gamma$, the right-angled Coxeter group $C(\Gamma)$ is defined by the presentation
$$\langle u \in V(\Gamma) \mid u^2=1 \ \text{for every $u \in V(\Gamma)$}, [u,v]=1 \ \text{if $\{u,v\} \in E(\Gamma)$} \rangle,$$
where $V(\Gamma)$ and $E(\Gamma)$ denote the vertex- and edge-sets of $\Gamma$. Right-angled Coxeter groups define a very interesting class of groups for several reasons. First, they have been used as a seminal source of (counter)examples \cite{RACGdiv, RACGhypdim, RACGhyOsadja, RACGboundaryMorse} and many groups turn out to embed into right-angled Coxeter groups \cite{MR2377497}. This motivates the idea that the family of right-angled Coxeter groups encompasses a large diversity of groups. And second, powerful geometric tools are available in order to study these groups, mainly their cubical geometry. (See Section \ref{section:cc} for more information.) 

\medskip \noindent
The main contribution of this article is the construction of an algorithmic solution to the embedding problem among two-dimensional right-angled Coxeter groups, i.e., right-angled Coxeter groups defined by triangle-free simplicial graphs.

\begin{thm}\label{thm:AlgoIntroduction}
There exists an algorithm determining, given two finite and triangle-free simplicial graphs $\Phi, \Psi$, whether or not $C(\Phi)$ is isomorphic to a subgroup of $C(\Psi)$. If so, an explicit basis is provided. Moreover, the algorithm also determines whether or not $C(\Phi)$ is isomorphic to a finite-index subgroup of $C(\Psi)$. If so, an explicit basis is provided and the index is computed.
\end{thm}

\noindent
Notice that the finite-index part of the statement is also proved in \cite[Theorem D]{RACGStallings}. 

\medskip \noindent
The techniques developed in this article in order to prove Theorem \ref{thm:AlgoIntroduction} can also be used to obtained explicit solutions to embedding problems among specific families of (two-dimensional) right-angled Coxeter groups. For instance, we prove:

\begin{thm}\label{thm1}
Let $\Gamma$ be a finite simplicial graph and let $C_n$ denote the cycle of length $n \geq 5$. Then $C(\Gamma)$ is isomorphic to a subgroup of $C(C_n)$ if and only if $\Gamma$ is either a disjoint union of segments or a single cycle of length $m$ such that $n-4$ divides $m-4$. 
\end{thm}

\noindent
In particular, $C(C_m)$ is isomorphic to a subgroup of $C(C_n)$ if and only if $n-4$ divides $m-4)$. (Compare with \cite[Theorem 1.12]{MR3039768} for right-angled Artin groups.)

\begin{thm}\label{thm2}
Let $R,S$ be two finite trees. Then $C(R)$ is isomorphic to a subgroup of $C(S)$ if and only if there exists a graph morphism $\varphi : R \to S$ which sends a vertex of degree $2$ to a vertex of degree $\geq 2$ and a vertex of degree $\geq 3$ to a vertex of degree $\geq 3$. 
\end{thm}

\noindent
See Examples \ref{ex:GraphMorphism} and \ref{ex:EmbeddingTree} for explicit illustrations of this statement.

\paragraph{First step towards the proof of Theorem \ref{thm:AlgoIntroduction}: reflection subgroups.} First of all, we conduct a general study of \emph{reflection subgroups} in right-angled Coxeter groups. Algebraically, a reflection in $C(\Gamma)$ is a conjugate of a generator. More geometrically, when looking at the action of $C(\Gamma)$ on its CAT(0) cube complex $X(\Gamma)$, a reflection stabilises each edge dual to some hyperplane $J$ and inverts the two halfspaces delimited by $J$. Our study is based on a simple ping-pong lemma we proved in \cite{Qm}. A particular case of our structure theorem is the following (see Theorem \ref{thm:ReflectionGeneral} for a more general statement):

\begin{thm}\label{thm:PP}
Let $\Gamma$ be a simplicial graph and $H \leq C(\Gamma)$ a subgroup generated by reflections. Let $\mathcal{J}$ denote the collection of the hyperplanes $J$ of $X(\Gamma)$ such that $H$ contains a reflection $r_J$ along it. Also, let $\mathcal{J}_0$ be the maximal peripheral subcollection of $\mathcal{J}$ and let $\Delta$ denote the crossing graph of $\mathcal{J}_0$ (i.e., the graph whose vertex-set is $\mathcal{J}_0$ and whose edges link two hyperplanes whenever they are transverse). Then the map 
$$\left\{ \begin{array}{lcl} \Delta & \mapsto & H \\ J & \mapsto & r_J \end{array} \right.$$
induces an isomorphism $C(\Delta) \to H$.
\end{thm}

\noindent
A collection of hyperplanes is \emph{peripheral} if no hyperplane of the collection separates another one from the vertex $1$. 

\medskip \noindent
Interestingly, (the general version of) Theorem \ref{thm:PP} leads to simple proofs of several results available in the literature. For instance, it is possible to (re)prove that reflection subgroups are word-quasiconvex \cite[Theorem B]{RACGStallings}, and that word-quasiconvex subgroups are separable \cite[Theorem A]{MR2413337}. We refer to Sections \ref{section:QC} and \ref{section:FI} for more details. 

\medskip \noindent
Theorem \ref{thm:PP} leads to the following notion of morphisms between right-angled Coxeter groups, which is fundamental in our work:

\begin{definition}
Let $\Phi,\Psi$ be two simplicial graphs. A morphism $C(\Phi) \to C(\Psi)$ is a \emph{peripheral embedding} if it sends the generators of $C(\Phi)$ to pairwise distinct reflections along a peripheral collection of hyperplanes of $X(\Psi)$.
\end{definition}

\paragraph{Second step towards the proof of Theorem \ref{thm:AlgoIntroduction}: morphisms between right-angled Coxeter groups.} Among right-angled Coxeter groups, we do not only study embeddings, but much more generally we are able to classify all the possible morphisms. Loosely speaking, we show that every morphism can be decomposed as a composition of elementary morphisms. More precisely:

\begin{thm}\label{thm:IntroMainMo}
Let $\Phi, \Psi$ be two finite graphs and $\rho : C(\Phi) \to C(\Psi)$ a morphism. Then there exist 
\begin{itemize}
	\item a sequence of graphs $\Lambda_1, \ldots, \Lambda_p$;
	\item a diagonal morphism $\delta : C(\Phi) \hookrightarrow C(\Lambda_1)$;
	\item partial conjugations $\alpha_1 \in \mathrm{Aut}(C(\Lambda_1)), \ldots, \alpha_{p-1} \in \mathrm{Aut}(C(\Lambda_{p-1}))$;
	\item foldings $\pi_1 : C(\Lambda_1) \to C(\Lambda_2), \ldots, \pi_{p-1} : C(\Lambda_{p-1}) \to C(\Lambda_p)$;
	\item and a peripheral embedding $\bar{\rho} : C(\Lambda_p) \hookrightarrow C(\Psi)$
\end{itemize}
such that $\rho = \bar{\rho} \circ \pi_{p-1} \circ \alpha_{p-1} \circ \cdots \circ \pi_1 \circ \alpha_1 \circ \delta$. 
\end{thm}

\noindent
In other words, we a commutative diagram
\[
\xymatrix{ C(\Phi) \ar[rrrr]^\rho \ar[dr]_\delta & & & & C(\Psi) \\ & C(\Lambda_1) \ar@(ul,ur)^{\alpha_1} \ar[dr]_{\pi_1} & & & \\ & & C(\Lambda_2) \ar@(ul,ur)^{\alpha_2} \ar[dr]_{\pi_2} & & \\ & & & \ddots \ar[dr]_{\pi_{p-1}} & \\ & & & & C(\Lambda_p) \ar[uuuu]_{\bar{\rho}}
 }
\]

\medskip \noindent
We refer to Sections \ref{section:ReMo} and \ref{section:Diagonal} for the definitions of foldings, partial conjugations and diagonal morphisms. If moreover $\Psi$ is triangle-free (i.e., if $C(\Psi)$ is two-dimensional), then the diagonal morphism can be replaced with a composition of even simpler morphisms, namely \emph{ingestions} and \emph{erasings}. (See Section \ref{section:MainThm} for details.) Such a description leads to the following classification of embeddings (generalising \cite[Theorem C]{RACGStallings}):

\begin{cor}\label{cor:IntroSub}
Let $\Phi,\Psi$ be two finite simplicial graphs and $\rho : C(\Phi) \hookrightarrow C(\Psi)$ an injective morphism. Assume that $\Psi$ is triangle-free and that $\Phi$ has no isolated vertex. There exists an automorphism $\alpha \in \mathrm{Aut}(C(\Phi))$ such that $\rho \circ \alpha : C(\Phi) \hookrightarrow C(\Psi)$ is a peripheral embedding. 
\end{cor}

\noindent
As an application, given two triangle-free simplicial graphs $\Phi$ and $\Psi$, it follows that $C(\Phi)$ is isomorphic to a subgroup of $C(\Psi)$ if and only if the cube complex $X(\Psi)$ contains a peripheral collection of hyperplanes whose crossing graph is isomorphic to $\Phi$.

\paragraph{Third step towards the proof of Theorem \ref{thm:AlgoIntroduction}: cuts-and-pastes.} However, the solution to the embedding problem provided by Corollary \ref{cor:IntroSub} is not algorithmic, as we need to check infinitely many configurations of hyperplanes in order to show that a given right-angled Coxeter group is not isomorphic to a subgroup of another one. We address this problem in Section \ref{section:peripheral}. Roughly speaking, we show that, given a peripheral collection of hyperplanes $\mathcal{J}$, if one hyperplane of $\mathcal{J}$ does cross a ball $B(1,\mathrm{cst}(\#\mathcal{J}))$, then it is possible to \emph{cut-and-paste} $\mathcal{J}$ in order to create a new peripheral collection of hyperplanes $\mathcal{J}'$, with the same crossing graph, and such that the sum of the distances from $1$ to the hyperplanes of $\mathcal{J}'$ is smaller than the same sum for $\mathcal{J}$. Therefore, by iterating, we are able to construct a peripheral collection of hyperplanes, with the same crossing graph, all of whose hyperplanes cross a given ball. The argument leads to the following statement:

\begin{thm}
Let $\Phi,\Psi$ be two finite simplicial graphs. Assume that $\Psi$ is triangle-free. Then $C(\Phi)$ is isomorphic to a subgroup of $C(\Psi)$ if and only if $C(\Psi)$ contains a basis of $C(\Phi)$ in the ball of radius $2(1+(1+2 \cdot \# V(\Phi)) \cdot \# V(\Psi) )$ centered at the vertex $1$.
\end{thm}

\noindent
Then, an algorithm solving the embedding problem among two-dimensional right-angled Coxeter groups follows easily.

\paragraph{Organisation of the article.} Section \ref{section:cc} is a preliminary section dedicated to the cubical geometry of right-angled Coxeter groups. Next, Section \ref{section:reflection} contains our general study of reflection groups acting on CAT(0) cube complexes: a general structure theorem is proved in Section \ref{section:pingpong}, namely the general version of Theorem \ref{thm:PP}, and Sections~\ref{section:QC}--\ref{section:FI} contain a few applications to right-angled Coxeter groups. In Sections~\ref{section:ReMo} and \ref{section:Diagonal}, we study morphisms between right-angled Coxeter groups of any dimension, proving Theorem \ref{thm:IntroMainMo}, and next we focus on the two-dimensional case in Section \ref{section:MainThm}, which contains the proof of Corollary \ref{cor:IntroSub}. A few applications, including Theorems \ref{thm1} and \ref{thm2}, are proved in Section \ref{section:Appli}. Cuts-and-pastes are studied in Section \ref{section:peripheral}, and Theorem~\ref{thm:IntroMainMo} is proved in Section \ref{section:AlgoFinal}.  The last section is dedicated to concluding remarks and a few open questions.

\paragraph{Acknowledgements.} I am grateful to Ivan Levcovitz and Sang-Hyun Kim for their comments on an earlier version of the article. This work was supported by a public grant as part of the Fondation Math\'ematique Jacques Hadamard.

\section{Cubical geometry of right-angled Coxeter groups}\label{section:cc}

\noindent
In this preliminary section, we recall basic definitions and properties related to CAT(0) cube complexes and right-angled Coxeter groups which will be used in the rest of the article. We begin with CAT(0) cube complexes.

\paragraph{CAT(0) cube complexes.}
A \textit{cube complex} is a CW complex constructed by gluing together cubes of arbitrary (finite) dimension by isometries along their faces. It is \textit{nonpositively curved} if the link of any of its vertices is a simplicial \textit{flag} complex (ie., $n+1$ vertices span a $n$-simplex if and only if they are pairwise adjacent), and \textit{CAT(0)} if it is nonpositively curved and simply-connected. See \cite[page 111]{MR1744486} for more information.

\medskip \noindent
Fundamental tools when studying CAT(0) cube complexes are \emph{hyperplanes}. Formally, a \textit{hyperplane} $J$ is an equivalence class of edges with respect to the transitive closure of the relation identifying two parallel edges of a square. Geometrically, a hyperplane $J$ is rather thought of as the union of the \textit{midcubes} transverse to the edges belonging to $J$ (sometimes referred to as its \emph{geometric realisation}). See Figure \ref{figure27}. The \textit{carrier} $N(J)$ of a hyperplane $J$ is the union of the cubes intersecting (the geometric realisation of) $J$. Two distinct hyperplanes are \emph{transverse} if their geometric realisations intersect, and they are \emph{tangent} if they are not transverse but their carriers intersect.
\begin{figure}
\begin{center}
\includegraphics[trim={0 13cm 10cm 0},clip,scale=0.4]{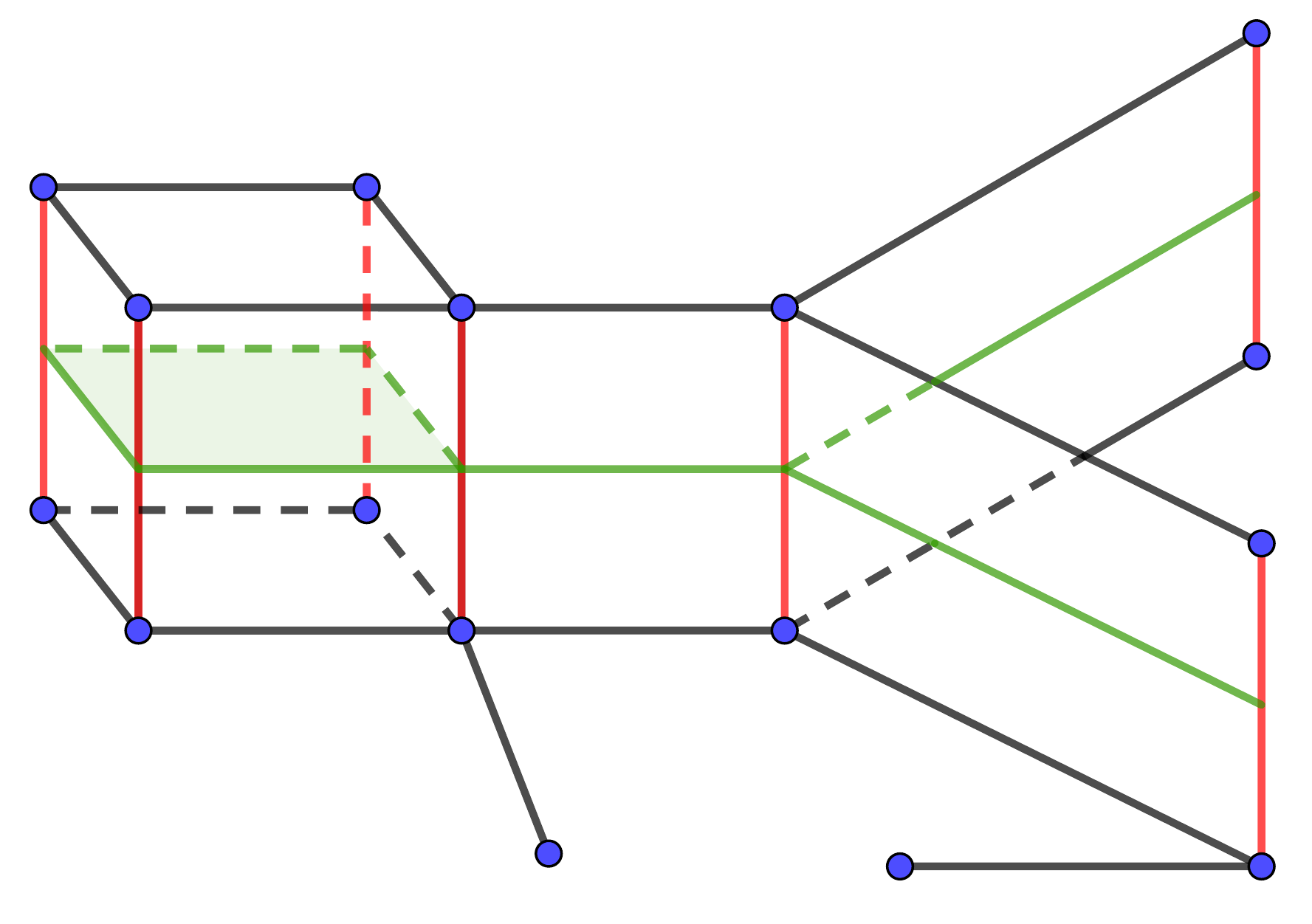}
\caption{A hyperplane (in red) and the associated union of midcubes (in green).}
\label{figure27}
\end{center}
\end{figure}

\medskip \noindent
There exist several metrics naturally defined on a CAT(0) cube complex. For instance, one can extend in a standard way the Euclidean metrics defined on each cube to a global length metric, and the distance one obtains in this way turns out to be CAT(0). However, in this article, we are mainly interested in the graph metric defined on the one-skeleton of the cube complex, referred to as its \emph{combinatorial metric}. Unless specified otherwise, we will always identify a CAT(0) cube complex with its one-skeleton, thought of as a collection of vertices endowed with a relation of adjacency. In particular, when writing $x \in X$, we always mean that $x$ is a vertex of $X$. 

\medskip \noindent
The following theorem will be often used along the article without mentioning it.

\begin{thm}\emph{\cite{MR1347406}}
Let $X$ be a CAT(0) cube complex.
\begin{itemize}
	\item If $J$ is a hyperplane of $X$, the graph $X \backslash \backslash J$ obtained from $X$ by removing the (interiors of the) edges of $J$ contains two connected components. They are convex subgraphs of $X$, referred to as the \emph{halfspaces} delimited by $J$.
	\item For every vertices $x, y \in X$, the distance between $x$ and $y$ coincides with the number of hyperplanes separating them.
\end{itemize}
\end{thm}

\medskip \noindent
Occasionally, we will use the \emph{$\ell^\infty$-metric}, which corresponds to the graph metric $d_\infty$ associated to the graph obtained from (the one-skeleton of) our CAT(0) cube complex by adding an edge between any two vertices which belong to a common cube. Alternatively, the $d_\infty$-distance between two vertices $x$ and $y$ can be defined as the maximal number of pairwise non-transverse hyperplanes separating $x$ and $y$ \cite{depth}. It is worth noticing the the $\ell^\infty$-metric is bi-Lipschitz equivalent to the combinatorial metric:  

\begin{lemma}\label{lem:dinfty}
Let $X$ be a finite-dimensional CAT(0) cube complex. For every vertices $x,y \in X$, the following inequality holds:
$$d_\infty(x,y) \leq d(x,y) \leq \dim(X) \cdot d_\infty(x,y).$$
\end{lemma}

\noindent
Interestingly, balls with respect to the $\ell^\infty$-metric turn out to be convex \cite[Corollary~3.5]{packing}.

\paragraph{Projections onto convex subcomplexes.} The projection of a vertex onto a convex subcomplex in a CAT(0) cube complex is defined by the following proposition:

\begin{prop}\emph{\cite[Lemma 13.8]{MR2377497}} 
Let $X$ be a CAT(0) cube complex, $C \subset X$ be a convex subcomplex and $x \in X \backslash C$ a vertex. There exists a unique vertex $y \in C$ minimising the distance to $x$. Moreover, for any vertex of $C$, there exists a geodesic from it to $x$ passing through $y$.
\end{prop}

\noindent
A particularly useful lemma is:

\begin{lemma}\label{lem:ProjSep}\emph{\cite[Lemma 13.8]{MR2377497}}
Let $X$ be a CAT(0) cube complex, $C \subset X$ a convex subcomplex and $x \in X$ a vertex. Every hyperplane separating $x$ from its projection onto $C$ separates $x$ from $C$.
\end{lemma}

\noindent
The following statement will be also needed:

\begin{lemma}\label{lem:InterProj}
Let $X$ be a CAT(0) cube complex and $Y,Z \subset X$ two convex subcomplexes. If $Y \cap Z \neq \emptyset$, then $\mathrm{proj}_Z(y) \in Y$ for every $y \in Y$. 
\end{lemma}

\begin{proof}
Let $y'$ (resp. $y''$) denote the projection of $y$ onto $Z$ (resp. $Y \cap Z$), and let $m$ denote the \emph{median point} of $y,y',y''$, i.e., the unique vertex of $X$ which belong to a geodesic between any two vertices among $y,y',y''$ \cite[Proposition 2.21]{MR2413337}. Because $m$ belongs to a geodesic between the two vertices $y,y''$ of $Y$, necessarily $m \in Y$. And because $m$ belongs to a geodesic between the two vertices $y',y''$ of $Z$, necessarily $m \in Z$. We conclude that $y' = m \in Y \cap Z$, as desired.
\end{proof}

\noindent
Our next lemma is an easy consequence of the fact that convex subcomplexes satisfy the \emph{Helly property}, i.e., if finitely many convex subcomplexes pairwise intersect then the total intersection is non-empty \cite[Corollary 2.22]{MR2413337}. 

\begin{lemma}\label{lem:DistInter}
Let $X$ be a CAT(0) cube complex and $Y_1, \ldots, Y_n \subset X$ pairwise intersecting convex subcomplexes. For every vertex $x \in X$, a hyperplane separates $x$ from $\bigcap\limits_{i=1}^n Y_i$ if and only if it separates $x$ from $Y_i$ for some $1 \leq i \leq n$.
\end{lemma}

\begin{proof}
Let $J$ be a hyperplane. Assume that, for every $1 \leq i \leq n$, $J$ does not separate $x$ from $Y_i$. Consequently, the halfspace $J^+$ delimited by $J$ which contains $x$ has a non-trivial intersection with $Y_i$. It follows from the Helly property that $J^+ \cap \bigcap\limits_{i=1}^n Y_i$ is non-empty, so that $J^+$ contains $x$ and a vertex of $\bigcap\limits_{i=1}^n Y_i$. In other words, $J$ does not separate $x$ from $\bigcap\limits_{i=1}^n Y_i$. Conversely, it is clear that if $J$ separate $x$ from $Y_i$ for some $1 \leq i \leq n$ then $J$ separates $x$ from $\bigcap\limits_{i=1}^n Y_i$. 
\end{proof}

\noindent
For our next lemma, recall that a \emph{reflection} in a CAT(0) cube complex is an isometry which stabilises and inverts all the edges dual to a given hyperplane.

\begin{lemma}\label{lem:shortendist}
Let $X$ be a CAT(0) cube complex, $J,H$ two hyperplanes and $x \in X$ a vertex. Assume that $J$ separates $x$ from $H$. If $r \in \mathrm{Isom}(X)$ is a reflection along $J$, then $d(x,N(rH))<d(x,N(H))$. 
\end{lemma}

\begin{proof}
We distinguish two cases. First, assume that $H$ separates $J$ and $rx$, or equivalently, that $rH$ separates $x$ and $J$. As $rH$ separates $x$ from $H$, it is clear that $d(x,N(rH))<d(x,N(H))$. Next, assume that $rx$ lies between $J$ and $H$. We claim that the inclusion
$$r \mathcal{W}(rx,N(H) \cup \{x\}) \subset \mathcal{W}(x, N(J))$$
holds, where $\mathcal{W}(S_1,S_2)$ denotes the collection of the hyperplanes of $X$ separating two subsets $S_1,S_2 \subset X$. So let $K$ be a hyperplane separating $rx$ from $N(H)$ and $x$. Notice that $K$ cannot be transverse to $J$, since otherwise $r$ would stabilise the halfspace delimited by $K$ which contains $rx$ and so $K$ would not separate $x$ and $rx$. Consequently, $K$ lies between $J$ and $K$. It follows that $rK$ lies in the halfspace delimited by $J$ which contains $x$. Moreover, since $K$ separates $rx$ from $J$, necessarily $rK$ separates $x$ from $rJ=J$, concluding the proof of our claim. Now, because
$$\mathcal{W}(x,N(H)) \subset \mathcal{W}(x,N(J)) \sqcup \{J \} \sqcup \mathcal{W}(\{x,rx\}, N(H)),$$
we deduce from our claim that
$$\begin{array}{lcl} d(x,N(H)) & = & \# \mathcal{W}(x,N(H)) \geq 1+ \# \mathcal{W}(rx,N(H) \cup \{x\}) + \# \mathcal{W}(\{x,rx\},N(H)) \\ \\ & \geq & 1+ \# \mathcal{W}(rx,N(H)) = 1+ d(rx,N(H))> d(rx,N(H)) \end{array}$$
concluding the proof of our lemma.
\end{proof}

\paragraph{Normal form in right-angled Coxeter groups.} Fixing a simplicial graph $\Gamma$, the right-angled Coxeter group $C(\Gamma)$ has a canonical generating set, namely the generating set coming from the presentation which defines it (which can be identified with the vertex-set of $\Gamma$). Below, we describe a normal form associated to this generating set. For more details, we refer to \cite{GreenGP} (see also \cite{GPvanKampen}) where the normal form is proved in a more general setting. 

\medskip \noindent
Clearly, the following operations on a word of generators $g_1 \cdots g_n$ do not modify the element of $C(\Gamma)$ it represents:
\begin{itemize}
	\item[(O1)] delete the letter $g_i=1$;
	\item[(O2)] if $g_i = g_{i+1} \in G$, remove the two letters $g_i$ and $g_{i+1}$;
	\item[(O3)] if $g_i$ and $g_{i+1}$ are adjacent vertices, switch them.
\end{itemize}
A word is \emph{reduced} if its length cannot be shortened by applying these elementary moves. Given a word $g_1 \cdots g_n$ and some $1\leq i<n$, if  $g_i$ is adjacent to each $g_{i+1}, \ldots, g_n$, then the words $g_1 \cdots g_n$ and $g_1\cdots g_{i-1} \cdot g_{i+1} \cdots g_n \cdot g_i$ represent the same element of $C(\Gamma)$; we say that $g_i$ \textit{shuffles to the right}. Analogously, one can define the notion of a letter shuffling to the left. If $g=g_1 \cdots g_n$ is a reduced word and $h$ is a generator, then a reduction of the product $gh$ is given by
\begin{itemize}
	\item $g_1 \cdots g_n$ if $h=1$;
	\item $g_1 \cdots g_{i-1} \cdot g_{i+1} \cdots g_n$ if $g_i=h$ and  $g_i$ shuffles to the right;
	\item $g_1 \cdots g_nh$ otherwise.
\end{itemize}
In particular, every element of $C(\Gamma)$ can be represented by a reduced word, and this word is unique up to applying the operation (O3). We will also need the following definitions:

\begin{definition}
Let $g \in C(\Gamma)$ be an element. The \emph{head} of $g$, denoted by $\mathrm{head}(g)$, is the collection of the first letters appearing in the reduced words representing $g$. Similarly, the \emph{tail} of $g$, denoted by $\mathrm{tail}(g)$, is the collection of the last letters appearing in the reduced words representing $g$. 
\end{definition}

\begin{definition}
Let $g,h \in C(\Gamma)$ be two elements. One says that $g$ is a \emph{prefix} of $h$ if one of the reduced words representing $h$ has a prefix representing $g$. 
\end{definition}

\paragraph{Cube complexes of RACG.} Given a simplicial graph $\Gamma$, let $X(\Gamma)$ denote the cube complex whose one-skeleton with the Cayley graph of $C(\Gamma)$ with respect to its canonical generating set and whose $n$-cubes fill in the subgraphs isomorphic to one-skeleta of $n$-cubes.

\begin{prop}\label{prop:RACGcc}
The cube complex $X(\Gamma)$ is a CAT(0) cube complex of dimension $\mathrm{clique}(\Gamma)$, the maximal size of a complete subgraph in $\Gamma$. 
\end{prop}

\noindent
We refer to \cite{DavisCoxeter} for more information. 

\medskip \noindent
First of all, we have the following structure of geodesics in $X(\Gamma)$, which follows from the very definition of $X(\Gamma)$:

\begin{lemma}\label{lem:geod}
Let $g,h \in C(\Gamma)$ be two elements. If $v_1 \cdots v_n$ is a reduced word representing $g^{-1}h$, then 
$$g, \ gv_1, \ gv_1v_2, \ldots, gv_1 \cdots v_n = h$$
defines a geodesic in $X(\Gamma)$ from $g$ to $h$. Moreover, every geodesic has this form. 
\end{lemma}

\noindent
Next, we would like to describe the hyperplanes of $X(\Gamma)$. Our first result in this direction is the following well-known lemma. (We refer the reader to \cite[Theorem 2.10, Proposition 2.11]{AutGP} for a proof in a more general setting.) 

\begin{lemma}\label{lem:DescriptionHyp}
Let $J$ be a hyperplane of $X(\Gamma)$. There exist $g \in C(\Gamma)$ and $u \in V(\Gamma)$ such that the edges dual to $J$ are 
$$\{(gh,ghu) \mid h \in \langle \mathrm{star}(u) \rangle \}.$$ 
As a consequence, the edges of $J$ are all labelled by the same vertex of $\Gamma$ and $\mathrm{stab}(J)= g \langle \mathrm{star}(u) \rangle g^{-1}$. 
\end{lemma}

\noindent
Because all the edges of a hyperplane are labelled by the same vertex of $\Gamma$, one says that this vertex \emph{labels} the hyperplane. It is worth noticing that two hyperplanes belong to the same $C(\Gamma)$-orbit if and only if they have the same label. For convenience, given a vertex $u \in V(\Gamma)$, we denote by $J_u$ the hyperplane dual to the edge $(1,u)$.  

\medskip \noindent
Our next two lemmas follow easily from the definition of squares in $X(\Gamma)$.

\begin{lemma}\label{lem:transverseimpliesadj}
Two transverse hyperplanes are labelled by adjacent vertices. 
\end{lemma}

\begin{lemma}\label{lem:transverseiff}
Two tangent hyperplanes are transverse if and only if they are labelled by adjacent vertices.
\end{lemma}

\noindent
Our next lemma provides a description of halfspaces in $X(\Gamma)$.

\begin{lemma}\label{lem:halfspaces}
Let $g,h \in C(\Gamma)$ be two elements and $u \in V(\Gamma)$ a vertex. Assume that the tail of $g$ does not contain any vertex of $\mathrm{star}(u)$. The hyperplane $gJ_u$ separates $1$ from $h$ if and only if $gu$ is a prefix of $h$. 
\end{lemma}

\begin{proof}
Because $N(gJ_u)= g \langle \mathrm{star}(u) \rangle$, for every $k \in N(gJ_u)$ we can write $k=g s$ for some $s \in \langle \mathrm{star}(u) \rangle$. Notice that, because the tail of $g$ does not contain any vertex of $\mathrm{star}(u)$, the product $gs$ is reduced, hence 
$$d(1,k) = |gs| = |g|+ |s| > |g| = d(1,g).$$
Therefore, $g$ turns out to be the projection of $1$ onto $N(J)$, and we deduce that $gu$ (i.e., the neighbor of $g$ which is separated from it by $gJ_u$) is the projection of $1$ onto the halfspace delimited by $gJ_u$ which does not contain $1$. Consequently, if $gJ_u$ separates $1$ from $h$, then there exists a geodesic from $1$ to $h$ passing through $gu$. We conclude from Lemma \ref{lem:geod} that $gu$ is a prefix of $h$.

\medskip \noindent
Conversely, assume that $gu$ is a prefix of $h$. So $h$ can be written as a reduced word $g_1 \cdots g_n \cdot u \cdot h_1 \cdots h_m$ where $g_1 \cdots g_n$ represents $g$. We know from Lemma \ref{lem:geod} that
$$1, \ g_1, \ g_1g_2, \ldots, g_1 \cdots g_n=g, \ gu, \ guh_1, \ guh_1h_2, \ldots, guh_1 \cdots h_m=h$$
defines a geodesic from $1$ to $h$. But this geodesic passes through the edge $(g,gu)$ which is dual to $gJ_u$. Therefore, $gJ_u$ has to separate $1$ and $h$.
\end{proof}

\noindent
We conclude this section with a last preliminary lemma:

\begin{lemma}\label{lem:CrossingSquare}
Let $J, H$ be two distinct hyperplanes of $X(\Gamma)$. If $\Gamma$ has girth $\geq 5$ then the projection of $N(J)$ onto $N(H)$ is a single square, a single edge or a single vertex. 
\end{lemma}

\begin{proof}
If $J$ and $H$ are transverse, then the projection of $N(J)$ onto $N(H)$ is $N(H) \cap N(J)$, which must be a square as $X(\Gamma)$ is two-dimensional. From now on, assume that $J$ and $H$ are not transverse. If the projection of $N(J)$ onto $N(H)$ is not a single cube, then there exist two non-transverse hyperplanes $A,B$ crossing the projection. Of course, $A$ and $B$ are transverse to $H$, but they are also transverse to $J$ as a consequence of \cite[Proposition 2.7]{Contracting}. It follows from \cite[Corollary 2.17]{coningoff} that $X(\Gamma)$ contains a vertex having a cycle of length four in its link, which contradicts the fact that $\Gamma$ has girth~$\geq 5$.
\end{proof}

\section{Reflection subgroups}\label{section:reflection}

\noindent
First of all, let us recall the definition of reflections in CAT(0) cube complexes: 

\begin{definition}
Let $X$ be a CAT(0) cube complex. A \emph{reflection} is an isometry $r$ of $X$ such that, for some hyperplane $J$, $r$ inverts each edge dual to $J$. A \emph{reflection group} is a subgroup of $\mathrm{Isom}(X)$ generated by finitely many reflections. 
\end{definition}

\noindent
For instance, given a simplicial graph $\Gamma$, the element $gug^{-1}$ of $C(\Gamma)$ defines a reflection of $X(\Gamma)$ along the hyperplane $gJ_u$ for every $u \in V(\Gamma)$ and $g \in C(\Gamma)$. It is worth noticing that, because $C(\Gamma)$ acts on $X(\Gamma)$ with trivial vertex-stabilisers, $gug^{-1}$ is the unique reflection of $C(\Gamma)$ along $gJ_u$. In other words, a hyperplane of $X(\Gamma)$ determines a unique reflection of $C(\Gamma)$.

\medskip \noindent
The goal of this first section is to understand reflection subgroups in right-angled Coxeter groups. More precisely, we begin by proving a general structure theorem in Subsection~\ref{section:pingpong}, and then we describe a few applications in the next three subsections.

\subsection{Playing ping-pong}\label{section:pingpong}

\noindent
This subsection is dedicated to the proof of a general structure theorem about groups acting on CAT(0) cube complexes with reflections. Before stating our theorem, we need to introduce the following definition:

\begin{definition}
Let $X$ be a CAT(0) cube complex and $x_0 \in X$ a basepoint. A collection of hyperplanes $\mathcal{J}$ is \emph{$x_0$-peripheral} if no hyperplane of $\mathcal{J}$ separates $x_0$ from another hyperplane of $\mathcal{J}$. 
\end{definition}

\noindent
When dealing with right-angled Coxeter groups and their canonical CAT(0) cube complexes, the basepoint will always be the vertex $1$, so that a $1$-peripheral collection of hyperplanes will be referred to as a \emph{peripheral} collection for short. 

\medskip \noindent
The main statement of this section is the following structure theorem:

\begin{thm}\label{thm:ReflectionGeneral}
Let $G$ be a group acting on a CAT(0) cube complex $X$ with trivial vertex-stabilisers. Assume that $\mathcal{J}$ is a $G$-invariant collection of hyperplanes such that $G$ contains a reflection $r_J$ along each hyperplane $J$ of $\mathcal{J}$. If $Y \subset X$ denotes the intersection of all the halfspaces which are delimited by hyperplanes of $\mathcal{J}$ and which contain a fixed basepoint $x_0 \in X$, then  
$$G = R \rtimes \mathrm{stab}(Y), \ \text{where} \ R = \langle r_J, \ J \in \mathcal{J} \rangle.$$ 
Moreover, $Y$ is a fundamental domain of $R \curvearrowright X$ and, if $\mathcal{J}_0$ denotes the maximal $x_0$-peripheral subcollection of $\mathcal{J}$ and if $\Delta$ denotes its crossing graph, then the map sending a vertex $J$ of $\Delta$ to the reflection $r_J$ of $R$ induces an isomorphism $C(\Delta) \to R$. 
\end{thm}

\noindent
Theorem \ref{thm:ReflectionGeneral} can be thought of as a generalisation of \cite[Theorem G]{MR2413337} (whose proof is written for \emph{cube complexes with faces}) and is essentially contained in the proof of \cite[Theorem 10.54]{Qm} (written for \emph{quasi-median graphs}). However, as it is not an immediate consequence of one of these two theorems and that it is written in a different formalism, we provide a self-contained proof for the reader's convenience. The key tool to prove Theorem \ref{thm:ReflectionGeneral} is the following ping-pong lemma, which is extracted from the proof of \cite[Proposition 8.44]{Qm}. 

\begin{prop}\label{prop:pingpong}
Let $G$ be a group acting on a set $X$, $\Gamma$ a simplicial graph and $\mathcal{H}= \{r_v \mid v \in V(\Gamma) \}$ a collection of order-two elements of $G$. Assume that $\mathcal{H}$ generates $G$ and that $r_u$ and $r_v$ commute if $u$ and $v$ are two adjacent vertices of $\Gamma$. Next, assume that there exist a collection $\{ X_v \mid v \in V(\Gamma) \}$ of subsets of $X$ and a point $x_0 \in X \backslash \bigcup\limits_{v \in V(\Gamma)} X_v$ satisfying:
\begin{itemize}
	\item if $u,v \in V(\Gamma)$ are adjacent, then $r_u \cdot X_u \subset X_u$;
	\item if $u,v \in V(\Gamma)$ are not adjacent and distinct, then $r_v \cdot X_u \subset X_v$;
	\item for every $u \in V(\Gamma)$, $r_u \cdot x_0 \in X_u$.
\end{itemize}
Then the map $u \mapsto r_u$ induces an isomorphism $C(\Gamma) \to G$. 
\end{prop}

\begin{proof}
It follows from our assumptions on $\mathcal{H}$ that the map $r \mapsto r_u$ induces a surjective morphism $C(\Gamma) \twoheadrightarrow G$. In order to show that this morphism is also injective, we want to prove the following claim: for any non-empty reduced word $w$ of $C(\Gamma)$, thought of as an element of $G$, $w \cdot x_0$ belongs to $X_u$ where $u$ is a vertex of $\Gamma$ which belongs to the head of $w$. Notice that, since $x_0 \notin \bigcup\limits_{v \in V(\Gamma)} X_v$ by assumption, this implies that $w \cdot x_0 \neq x_0$, so that $w \neq 1$ in $G$. Proving this claim is sufficient to conclude the proof of our proposition. Also, an immediate consequence of the claim is the following fact, which we record for future use:

\begin{fact}\label{fact:pingpong}
For every non-trivial $g \in G$, we have $g \cdot x_0 \in \bigcup\limits_{u \in V(\Gamma)} X_u$.
\end{fact}

\noindent
So let us turn to the proof of our claim. We argue by induction on the length of $w$. If $w$ has length one, then $w =r_u$ for some $u \in V(\Gamma)$ and our third assumption implies that $w \cdot x_0 \in X_u$. Next, suppose that $w$ has length at least two. Write $w=r_uw'$ where $r_u$ is the first letter of $w$ and $w'$ the rest of the word. We know from our induction hypothesis that $w' \cdot x_0 \in X_v$ where $v$ is a vertex of $\Gamma$ which belongs to the head of $w'$. Notice that $u \neq v$ since otherwise the word $r_uw'$ would not be reduced. Two cases may happen: either $u$ and $v$ are not adjacent, so that our second assumption implies that $w \cdot x_0 \in r_u \cdot X_v \subset X_u$; or $u$ and $v$ are adjacent, so that our first assumption implies that $w \cdot x_0 \in r_u \cdot X_v \subset X_v$. It is worth noticing that, in the former case, $u$ clearly belongs to the head of $w$ since $r_u$ belongs to the head of $w$; in the latter case, we can write $w'$ as a reduced word $r_vw''$ and we have
$$w = r_uw'= r_ur_vw'' = r_vr_uw'',$$
so $r_v$ also belongs to the head of $w$, so that $v$ belongs to the head of $w$. This concludes the proof.
\end{proof}

\noindent
We are now ready to prove Theorem \ref{thm:ReflectionGeneral}.

\begin{proof}[Proof of Theorem \ref{thm:ReflectionGeneral}.]
We begin our proof by two elementary observations.

\begin{claim}
For every $J \in \mathcal{J}$, $r_J$ is an order-two element of $G$.
\end{claim}

\noindent
Because $r_J^2$ fixes pointwise each edge dual to $J$, it follows that $r_J^2=1$ from the fact that vertex-stabilisers are trivial.

\begin{claim}\label{claim:normal}
For every $g \in G$ and $J \in \mathcal{J}$, we have $gr_Jg^{-1}= r_{gJ}$.
\end{claim}

\noindent
Notice that $gr_Jg^{-1}$ and $r_{gJ}$ are two reflections of $G$ along the same hyperplane, namely $gJ$. Consequently, $gr_Jg^{-1}r_{gJ}$ fixes pointwise each edge dual to $gJ$, and because vertex-stabilisers are trivial, we conclude that $r_{gJ}=gr_Jg^{-1}$. 

\medskip \noindent
Notice that, as a consequence of Claim \ref{claim:normal}, the subgroup $R$ is normal in $G$. 

\medskip \noindent
Let $g \in G$. Fix some $r \in R$ and suppose that $rg \cdot x_0 \notin Y$. Then there exists some $J \in \mathcal{J}$ which separates $rg \cdot x_0$ from $Y$. Let $a \in N(J)$ denote the projection of $rg \cdot x_0$ onto $N(J)$ and $e$ the edge dual to $J$ containing $a$. Notice that
$$\begin{array}{lcl} d(x_0,rg \cdot x_0) & = & d(rg \cdot x_0,a)+d(a,x_0) = d(r_Jrg \cdot x_0, r_J \cdot a) + d(x_0, r_J \cdot a)+1 \\ \\ & \geq & d(x_0, r_Jrg \cdot x_0)+1 \end{array}$$
Thus, if we choose some $r \in R$ such that
$$d(rg \cdot x_0,g \cdot x_0)= \min \{ d(sg \cdot x_0,g \cdot x_0) \mid s \in R \},$$ 
we deduce from the previous observation that $rg \cdot x_0 \in Y$. On the other hand, $G$ permutes the connected components of $X$ cut along the hyperplanes of $\mathcal{J}$, and $Y$ is precisely the connected component which contains $x_0$, so $rg \in \mathrm{stab}(Y)$. Therefore,
$$g \in r^{-1} \cdot \mathrm{stab}(Y) \subset R \cdot \mathrm{stab}(Y).$$
Thus, we have proved that $G= R \cdot \mathrm{stab}(Y)$. 

\medskip \noindent
Next, we want to apply Proposition \ref{prop:pingpong} in order to prove that $R$ is isomorphic to the right-angled Coxeter group $C(\Delta)$. 

\medskip \noindent
The first point to verify is that 

\begin{claim}
The equality $R= \langle r_J, \ J \in \mathcal{J}_0 \rangle$ holds. 
\end{claim}

\noindent
Let $J_1 \in \mathcal{J}$. We want to prove that there exists some $r \in \langle r_J, \ J \in \mathcal{J}_0 \rangle$ such that $rJ_1 \in \mathcal{J}_0$, which is sufficient to deduce the previous equality. 

\medskip \noindent
Fix some $r \in \langle r_J, \ J \in \mathcal{J}_0 \rangle$, and suppose that $rJ_1 \notin \mathcal{J}_0$. Let $y_0$ denote the projection of $x_0$ onto $N(rJ_1)$. If no hyperplane of $\mathcal{J}$ separates $x_0$ and $y_0$, then $\mathcal{J}_0 \cup \{ rJ \}$ defines a new $x_0$-peripheral subcollection of $\mathcal{J}$, contradicting the maximality of $\mathcal{J}_0$. Therefore, there exists some hyperplane $J_2 \in \mathcal{J}$ separating $x_0$ and $y_0$. As a consequence of Lemma~\ref{lem:ProjSep}, $J_2$ also separates $rJ_1$ and $x_0$. Notice that, if $J_2 \notin \mathcal{J}_0$, then similarly there exists a third hyperplane separating $J_2$ and $x_0$, and so on. Since there exist only finitely many hyperplanes separating $rJ_1$ and $x_0$, we can suppose without loss of generality that $J_2 \in \mathcal{J}_0$. For convenience, set $s = r_{J_2}$. Notice that
$$\begin{array}{lcl} d(x_0,N(rJ_1)) & \geq & d(N(rJ_1),N(J_2)) + d(x_0,N(J_2)) +1 \\ \\ & \geq & d(N(srJ_1),N(J_2)) + d(x_0,N(J_2))+1 \\ \\ & \geq & d(x_0,N(srJ_1)) +1 \end{array}$$
Therefore, if we choose $r$ so that
$$d(x_0,N(rJ_1))= \min \left\{ d(x_0,N(sJ_1)) \mid s \in \langle  r_J, \ J \in \mathcal{J}_0 \rangle \right\},$$
then $rJ_1 \in \mathcal{J}_0$. This concludes the proof of our claim.

\medskip \noindent
Next, notice that 

\begin{claim}\label{Claim:TransverseCommute}
If $J,H \in \mathcal{J}$ are two transverse hyperplanes, then $r_J$ and $r_H$ commute.
\end{claim}

\noindent
Fix a square crossed by both $J$ and $H$, and let $a,b$ denote two of its opposite vertices. Then $r_Jr_H \cdot a = c = r_Hr_J \cdot a$, so $[r_J,r_H]$ fixes the vertex $a$. Because vertex-stabilisers are trivial, it follows that $r_J$ and $r_H$ commute, concluding the proof of the claim.

\medskip \noindent
Now, for every hyperplane $J \in \mathcal{J}_0$, let $X_J$ denote the halfspace delimited by $J$ which does not contain $x_0$. We want to prove that our collection of sets satisfies the three conditions of Proposition~\ref{prop:pingpong}. 

\medskip \noindent
The third condition is clear, and the second condition is a consequence of the fact that $\mathcal{J}_0$ is $x_0$-peripheral. So we only need to verify the first condition.

\begin{claim}\label{claim:HalfStab}
If $J,H \in \mathcal{J}_0$ are transverse, then $r_J \cdot X_H= X_H$.
\end{claim}

\noindent
Fix a square $C$ crossed by both $J$ and $H$, and let $a,b,c,d$ denote its vertices so that $a,c$ are not adjacent but are both adjacent to $b$ and $c$. Assume that the edge $[a,d]$ is dual to $J$ and is contained in $X_H$. By noticing that $X_H$ contains exactly the vertices of $X$ whose projections onto $C$ are either $a$ or $d$, it follows that
$$r_J \cdot X_H= \{ r_J \cdot x \in X \mid \mathrm{proj}_C(x) \in \{a,d\} \}= \{ x \in X \mid \mathrm{proj}_C(r_J \cdot x) \in \{a,d\} \}.$$
But, because $r_J C= C$, we have $\mathrm{proj}_C(r_J \cdot x) = \mathrm{proj}_{r_JC}(r_J \cdot x)= \mathrm{proj}_C(x)$ for every $x \in X$, hence $r_J \cdot X_H=X_H$, concluding the proof of our claim.

\medskip \noindent
Thus, we have proved that Proposition \ref{prop:pingpong} applies, showing that the map $J \mapsto r_J$ induces an isomorphism $C(\Delta) \to R$.

\medskip \noindent
Because the hypotheses of Proposition \ref{prop:pingpong} hold, we also know from Fact \ref{fact:pingpong} that, if $g \in R$ is non-trivial, then $g \cdot x_0 \in X_{J}$ for some $J \in \mathcal{J}$. But $g$ permutes the connected components of $X$ cut along the hyperplanes of $\mathcal{J}$, $Y$ is precisely the connected component which contains $x_0$, and $X_J$ is a union of connected components disjoint from $Y$, so $g \cdot Y \cap Y= \emptyset$. As a consequence, $g$ does not stabilise the connected component $Y$, i.e., $g \notin \mathrm{stab}(Y)$. Thus, we have proved that $R \cap \mathrm{stab}(Y)= \{ 1 \}$. This concludes the proof of the decomposition $G = R \rtimes \mathrm{stab}(Y)$.

\medskip \noindent
Finally, it remains to show that $Y$ is a fundamental domain for the action $R \curvearrowright X$. We saw in the previous paragraph that $rY \cap Y= \emptyset$ for every non-trivial $r \in R$, so no two vertices of $Y$ belong to the same $R$-orbit. Now, let $x \in X$ be an arbitrary vertex. Fix a geodesic $[x_0,x]$ from $x_0$ to $x$, and let $Z_0=Y, Z_1, \ldots, Z_k$ denote the sequence of connected components of $X$ cut along $\mathcal{J}$ which are successively crossed by $[x_0,x]$. Notice that, for every $0 \leq i \leq k-1$, $Z_i$ and $Z_{i+1}$ are separated by a unique hyperplane of $\mathcal{J}$, so that there exists some reflection $r_i \in R$ such that $r_i Z_{i+1} = Z_i$. It follows that $r_0 \cdots r_{k-1} \cdot Z_k =Z_0$. In other words, the element $r:= r_0 \cdots r_{k-1}$ of $R$ sends the connected component $Z_k$ containing $x$ to the connected component $Z_0=Y$. Thus, we have proved that any vertex of $X$ has an $R$-translate which belongs to $Y$. 
\end{proof}

\subsection{Quasiconvexity of reflection subgroups}\label{section:QC}

\noindent
In this subsection, we apply Theorem \ref{thm:ReflectionGeneral} in order to prove that reflection subgroups are ``nicely embedded''. More precisely, given a finite simplicial graph $\Gamma$, a subgroup $H \leq C(\Gamma)$ is \emph{word-quasiconvex} if there exists some $K \geq 0$ such that any geodesic in the Cayley graph of $C(\Gamma)$, constructed from its canonical generating set, between two points of $H$ lies in $K$-neighborhood of $H$. Notice that, as a consequence of \cite[Theorem~H]{MR2413337}, a subgroup of $C(\Gamma)$ is word-quasiconvex if and only if it is \emph{convex-cocompact}, i.e., it acts cocompactly on a convex subcomplex of $X(\Gamma)$. Thanks to Theorem \ref{thm:ReflectionGeneral}, we are able to give a short proof of \cite[Theorem B]{RACGStallings}, which states that reflection subgroups are word-quasiconvex:

\begin{thm}\label{thm:ReflectionGroups}\emph{\cite{RACGStallings}}
Let $\Gamma$ be a finite simplicial graph. A reflection subgroup of $C(\Gamma)$ is convex-cocompact and is a virtual retract.
\end{thm}

\noindent
Recall that, given a group $G$, a subgroup $H \leq G$ is a \emph{virtual retract} if there exists a finite-index subgroup $K \leq G$ containing $H$ and a morphism $r : K \to H$ (a \emph{retraction}) such that $r(h)=h$ for every $h \in H$. 

\medskip \noindent
We begin by proving a preliminary lemma:

\begin{lemma}\label{lem:UnionConvex}
Let $X$ be a CAT(0) cube complex, $Y \subset X$ a convex subcomplex, $J$ a hyperplane and $r \in \mathrm{Isom}(X)$ a reflection along $J$. If $J$ crosses $Y$ then $Y \cup rY$ is convex. 
\end{lemma}

\begin{proof}
First of all, let us notice that 

\begin{claim}\label{claim:Neighbor}
For every $x \in Y \cap N(J)$, the neighbor $y$ which is separated from $x$ by $J$ has to belong to $Y$. 
\end{claim}

\noindent
Indeed, as $J$ crosses $Y$, there exists a vertex $z \in Y$ such that $J$ separates $z$ and $x$. Then the concatenation of a geodesic $[z,y]$ from $z$ to $y$ with the edge between $y$ and $x$ does not cross twice a hyperplane (because $J$ does not cross $[z,y]$ by convexity of halfspaces), and so defines a geodesic. The convexity of $Y$ implies that $y$ must belong to $Y$, concluding the proof of our claim.

\medskip \noindent
Now let $a \in rY$ and $b \in Y$ be two vertices, and $[a,b]$ a geodesic between $a$ and $b$. Suppose for contradiction that there exists a vertex $c \in [a,b]$ which does not belong to $rY \cup Y$. Up to replacing $Y$ with $rY$, assume that $c \notin Y$. As a consequence of Lemma~\ref{lem:ProjSep}, there exists a hyperplane $H$ separating $c$ from $Y$. 

\medskip \noindent
Let $a'$ and $b'$ denote respectively the projections of $a$ and $b$ onto $N(J)$. As a consequence of Lemma \ref{lem:InterProj}, $a' \in rY$ and $b' \in Y$. Because $ra' \in Y$, it follows from Claim \ref{claim:Neighbor} that $a' \in Y$. Therefore, $H$ cannot separate $a'$ and $b'$ since it does not cross $Y$. For the same reason, $H$ cannot separate $b'$ and $b$ as a consequence of Lemma \ref{lem:ProjSep}. Finally, if $H$ separates $a$ and $a'$, then it separates $\{a,c\}$ and $\{b,b'\}$, and so it must be transverse to $J$, contradicting Lemma \ref{lem:ProjSep}. We conclude that $H$ does not separate $a$ and $b$.

\medskip \noindent
This contradicts the fact that $H$ separates $c$ and $b$ and that $c$ belongs to a geodesic between $a$ and $b$. 
\end{proof}

\begin{proof}[Proof of Theorem \ref{thm:ReflectionGroups}.]
Let $H \leq C(\Gamma)$ be a reflection subgroup and $\mathcal{J}$ a finite collection of hyperplanes such that $H$ is generated by reflections along hyperplanes of $\mathcal{J}$. As a consequence of Theorem \ref{thm:ReflectionGeneral}, we may suppose that $\mathcal{J}$ is peripheral. 

\medskip \noindent
Let $Y \subset X(\Gamma)$ denote the intersection of all the halfspaces containing $1$ which are delimited by hyperplanes of $\mathcal{J}$. Fix an $R \geq 1$ big enough so that every hyperplane of $\mathcal{J}$ intersects the ball $B(1,R)$, and let $\mathcal{H}$ denote the set of all the hyperplanes $J$ crossing $Y$ and satisfying $d(1,N(J))=R$. Notice that, as $X(\Gamma)$ is locally finite, necessarily $\mathcal{H}$ is finite. Moreover, $\mathcal{J}^+:= \mathcal{J} \cup \mathcal{H}$ is peripheral, and the intersection $Z$ of all the halfspaces containing $1$ which are delimited by hyperplanes of $\mathcal{J}^+$ is finite. 

\medskip \noindent
Let $H^+ \leq C(\Gamma)$ denote the subgroup generated by the reflections along the hyperplanes of $\mathcal{J}^+$. Clearly, $H^+$ contains $H$. Moreover, it follows from Theorem \ref{thm:ReflectionGeneral} that $Z$ is the fundamental domain of $H^+ \curvearrowright X(\Gamma)$, so $H^+$ must have finite index in $C(\Gamma)$. Also, if $\Delta$ denotes the crossing graph of $\mathcal{J}^+$, then $H^+$ is naturally isomorphic to the right-angled Coxeter group $C(\Delta)$ and the image of $H$ in $C(\Delta)$ is the subgroup generated by the subgraph of $\Delta$ corresponding to the crossing graph of $\mathcal{J}$, so that one gets a retraction $H^+ \to H$ by fixing the generators of $C(\Delta)$ corresponding to hyperplanes of $\mathcal{J}$ and by sending to $1$ the generators corresponding to hyperplanes of $\mathcal{J}^+ \backslash \mathcal{J}$. Thus, we have proved that $H$ is a virtual retract.

\medskip \noindent
Now, let $W$ be any finite convex subcomplex crossed by all the hyperplanes of $\mathcal{J}$. (For instance, take $W$ as the ball centered at $1$ of radius $R$ with respect to the $\ell^\infty$-metric.) By applying Lemma \ref{lem:UnionConvex} iteratively, we know that $\bigcup\limits_{h \in H} h \cdot W$ is convex. As $H$ clearly acts on it cocompactly, we conclude that $H$ is convex-cocompact, as desired. 
\end{proof}

\subsection{Embedding right-angled Coxeter groups}

\noindent
Embeddings between right-angled Coxeter groups can be constructed thanks to Theorem~\ref{thm:ReflectionGeneral}. Indeed, if $\Phi,\Psi$ are two simplicial graphs and if $X(\Psi)$ contains a peripheral collection of hyperplanes whose crossing graph is $\Phi$, then it follows from Theorem~\ref{thm:ReflectionGeneral} that the subgroup of $C(\Psi)$ generated by the reflections along these hyperplanes is isomorphic to $C(\Phi)$. This subsection is dedicated to examples of such constructions. We will see in Section \ref{section:Appli} that the converses of Corollary \ref{cor:EmbeddingExCycle}, Proposition \ref{prop:EmbeddingExTree} and Corollary \ref{cor:EmbeddingExSegment} hold. 

\begin{prop}\label{prop:double}
Let $\Gamma$ be a simplicial graph and $u \in V(\Gamma)$ a vertex. Let $\Psi$ denote the graph obtained by gluing two copies of $\Gamma \backslash \{u\}$ along $\mathrm{link}(u)$. Then $C(\Psi)$ is isomorphic to a subgroup of index two of $C(\Gamma)$. 
\end{prop}

\noindent
Notice that our proposition is also a consequence of \cite[Lemma 21]{MR3010817}. Below, we give a short geometric proof based on Theorem \ref{thm:ReflectionGeneral}.

\begin{proof}[Proof of Proposition \ref{prop:double}.]
Set $\mathcal{J}= \{ J_v \mid v \in V(\Gamma) \backslash \{u\} \}$. We claim that $\mathcal{J} \cup u \mathcal{J}$ is a peripheral collection of hyperplanes whose crossing graph is isomorphic to $\Psi$, which is sufficient to conclude according to Theorem \ref{thm:ReflectionGeneral}. Notice that every hyperplane of $\mathcal{J}$ is adjacent to the vertex $1$, so $\mathcal{J}$ must be peripheral as well as $u \mathcal{J}$. If $\mathcal{J} \cup u \mathcal{J}$ is not peripheral, then a hyperplane of $\mathcal{J}$ has to separate $1$ from a hyperplane of $u \mathcal{J}$. As every hyperplane of $u \mathcal{J}$ is adjacent to the vertex $u$, our hyperplane of $\mathcal{J}$ has to separate $1$ from $u$, which is impossible as $J_u$ is the unique hyperplane separating $1$ and $u$ but it does not belong to $\mathcal{J}$. So $\mathcal{J} \cup u \mathcal{J}$ is peripheral. Next, notice that
$$u\mathcal{J}= \{ J_v \mid v \in \mathrm{link}(u) \} \sqcup \{ uJ_v \mid v \in \Gamma \backslash \mathrm{star}(u) \}.$$
Moreover, if $v,w \in \Gamma \backslash \mathrm{star}(u)$ then $J_v$ and $uJ_w$ are not transverse. Otherwise, there would exist a path from $1$ to $u$ decomposing as the concatenation of a path in $N(J_v)= \langle \mathrm{star}(v) \rangle$ with a path in $N(uJ_w)=u \langle \mathrm{star}(w) \rangle u^{-1}$, hence
$$u \in \langle \mathrm{star}(v) \rangle \cdot u \langle \mathrm{star}(w) \rangle u^{-1},$$
which is not possible as a consequence of the normal formed described in Section \ref{section:cc}. We conclude that the crossing graph of $\mathcal{J} \cup u \mathcal{J}$ is indeed isomorphic to $\Psi$. 
\end{proof}

\noindent
In the next statement, for every $n \geq 3$ we denote by $C_n$ the cycle of length $n$. 

\begin{cor}\label{cor:EmbeddingExCycle}
For every $p,q \geq 5$, the right-angled Coxeter group $C(C_p)$ is isomorphic to a subgroup of $C(C_q)$ if $q-4$ divides $p-4$. 
\end{cor}

\begin{proof}
Assume that $q-4$ divides $p-4$, i.e., $p=k(q-4)+4$ for some $k \geq 1$. As a consequence of Proposition \ref{prop:double}, $C(C_{2s-4})= C(C_{2(s-4)+4})$ is isomorphic to a subgroup of $C(C_s)$ for every $s \geq 5$. Therefore, by applying Proposition \ref{prop:double} $k$ times, we conclude that $C(C_p)$ embeds into $C(C_q)$. 
\end{proof}

\noindent
The next proposition is the main result of this section.

\begin{prop}\label{prop:EmbeddingExTree}
Let $R,S$ be two finite trees. Assume that there exists a graph morphism $\varphi : R \to S$ which sends a vertex of degree $2$ to a vertex of degree $\geq 2$ and a vertex of degree $\geq 3$ to a vertex of degree $\geq 3$. Then $C(R)$ is isomorphic to a subgroup of $C(S)$. 
\end{prop}

\begin{proof}
We claim that $X(S)$ contains a peripheral collection of hyperplanes $\mathcal{J}$ whose crossing graph is $R$ and such that the hyperplane of $\mathcal{J}$ corresponding to a vertex $u \in V(R)$ is labelled in $C(S)$ by $\varphi(u)$. We argue by induction over the number of vertices of $R$.

\medskip \noindent 
If $R$ has at most one vertex, the conclusion is clear. If all the vertices of $R$ are leaves, then $R$ must be a single edge $[a,b]$ and the collection $\{J_{\varphi(a)},J_{\varphi(b)}\}$ allows us to conclude.

\medskip \noindent
From now on, assume that $R$ is a finite tree with at least two vertices, one of them not being a leaf. Fix a vertex $x \in V(R)$ which is adjacent to at least one leaf and adjacent to exactly one vertex $y \in V(R)$ which is not a leaf. (In other words, $x$ is a leaf in the tree obtained from $R$ by removing all the leaves.) Let $R_0$ denote the tree obtained from $R$ by removing all the leaves adjacent to $x$. By induction, we know that there exists a peripheral collection of hyperplanes $\mathcal{J}$ whose crossing graph is $R_0$ and such that the hyperplane of $\mathcal{J}$ corresponding to a vertex $u \in V(R_0)$ is labelled in $C(S)$ by $\varphi(u)$. Let $J,H \in \mathcal{J}$ denote the hyperplanes corresponding respectively to the vertices $x,y$ of $R_0$, and let $\partial J$ denote the intersection of $N(J)$ with the halfspace delimited by $J$ which contains $1$. Up to translating $\mathcal{J}$ by an element of $C(S)$, we suppose without loss of generality that $1 \in \partial J \cap N(H)$. Notice that $H$ is the unique hyperplane of $\mathcal{J}$ transverse to $J$. We distinguish two cases.

\medskip \noindent
First, assume that $x$ has degree two in $R$. Let $z$ denote the neighbor of $x$ distinct from $y$. By assumptions, $\varphi(x)$ has degree at least two. 

\medskip \noindent
If $\varphi(y) \neq \varphi(z)$, then $\langle \varphi(y), \varphi(z) \rangle$ is a bi-infinite line in $N(J)$ crossed by $H$, so that it contains a subray $r$ in $\partial J$. It follows from Lemma \ref{lem:CrossingSquare} that the projection of a hyperplane of $\mathcal{J} \backslash \{J\}$ onto $\partial J$ is either a single vertex or a single edge. Therefore, if $e \subset r$ is an edge labelled by $\varphi(z)$ which is sufficiently far away from $1$, then the hyperplane $E$ dual to $e$ is not transverse to any hyperplane of $\mathcal{J} \backslash \{J\}$ and does not separate $1$ from any of these hyperplanes. In other words, $\mathcal{J} \cup \{E\}$ is the desired peripheral collection with $R= R_0 \cup \{z\}$ as its crossing graph. 

\medskip \noindent
Otherwise, if $\varphi(y)=\varphi(z)$, then because $\varphi(x)$ has degree at least two there must exist a vertex $w \in V(S)$ which is adjacent to $\varphi(x)$ but distinct from $\varphi(z)$. As before, $\langle \varphi(z),w \rangle= \langle \varphi(y),w \rangle$ is a bi-infinite line in $N(J)$ crossed by $H$. By reproducing the previous argument word for word, we conclude that there exists a hyperplane $E$ labelled by $\varphi(z)$ such that $\mathcal{J}\cup \{E\}$ is the desired peripheral with $R$ as its crossing graph.

\medskip \noindent
Next, assume that $x$ has degree at least three in $R$. Let $z_1, \ldots, z_N$ denote the neighbors of $x$ distinct from $y$. By assumptions, $\varphi(x)$ has degree at least three. So there exists a neighbor $w$ of $\varphi(x)$ such that $\{\varphi(y),\varphi(z_1), \ldots, \varphi(z_N),w\}$ has cardinality at least three. (If $\{\varphi(y), \varphi(z_1), \ldots, \varphi(z_N)\}$ has cardinality at least three, set $w= \varphi(y)$; and otherwise, set $w$ has a neighbor of $x$ which does not belong to $\{\varphi(y), \varphi(z_1), \ldots, \varphi(z_N) \}$.) Then $\langle \varphi(y), \varphi(z_1), \ldots, \varphi(z_N),w \rangle$ is a $k$-regular tree ($k\geq 3$) which is contained in $N(J)$ and crossed by $H$ along a single edge $e$. Fix an integer $D \geq 1$ and $N$ pairwise distinct edges $e_1, \ldots, e_N \subset \partial J$ such that:
\begin{itemize}
	\item these edges lie in the connected component of $\partial J \backslash \{e\}$ which contains the projection $p$ of $1$ onto $N(J)$;
	\item for every $1 \leq i \leq N$, $e_i$ is labelled by $\varphi(z_i)$;
	\item our edges lie outside the ball $B(p,D)$;
	\item $e_i$ does not separate $p$ from $e_j$ for every $1 \leq i,j \leq N$;
\end{itemize}
It follows from Lemma \ref{lem:CrossingSquare} that the projection of a hyperplane of $\mathcal{J} \backslash \{J\}$ onto $\partial J$ is either a single vertex or a single edge. Therefore, if $D$ is chosen sufficiently large compared to the cardinality of $\mathcal{J}$, then an $e_i$ cannot separate $p$ from a point of the projection onto $\partial J$ of a hyperplane of $\mathcal{J} \backslash \{J\}$. Consequently, if we denote by $J_i$ the hyperplane dual to $e_i$ for every $1 \leq i \leq N$, the collection $\mathcal{J} \cup \{J_i, 1 \leq i \leq N\}$ is peripheral and its crossing graph is a tree obtained from $R_0$ by adding $N$ neighbors to $x$, i.e., the crossing graph is $R$ as desired.

\medskip \noindent
This concludes the proof of our claim, and we deduce from it and from Theorem \ref{thm:ReflectionGeneral} that $C(R)$ is isomorphic to a subgroup of $C(S)$.
\end{proof}

\noindent
Below, we record a few easy consequences of Proposition \ref{prop:EmbeddingExTree}.

\begin{cor}\label{cor:ForestFromTree}
Let $F$ be a finite forest and $T$ a finite tree which contains at least three vertices. Let $T_1, \ldots, T_k$ denote the components of $F$. Then $C(F)$ is isomorphic to a subgroup of $C(T)$ if and only if $C(T_i)$ is isomorphic to a subgroup of $C(T)$ for every $1 \leq i \leq k$.
\end{cor}

\begin{proof}
The ``only if'' direction is clear. Conversely, assume that $C(T_i)$ embeds into $C(T)$ for every $1 \leq i \leq k$. The desired conclusion is an immediate consequence of the following statement:

\begin{claim}
For every $k \geq 1$, $C(T)$ contains a subgroup isomorphic to the free product of $k$ copies of itself. 
\end{claim}

\noindent
Because $T$ contains at least three vertices, there exist two adjacent vertices $a,b \in T$ such that $b$ has degree at least two. Let $T_+$ denote the tree obtained gluing $T$ and a path $[x,y,z]$ of length two by identifying $z$ and $a$. Then
$$\left\{ \begin{array}{ccc} T_+ & \to & T \\ u & \mapsto & \left\{ \begin{array}{cl} u & \text{if $u \in T$} \\ b & \text{if $u=y$} \\ a & \text{if $u=x$} \end{array} \right. \end{array} \right.$$
defines a graph morphism. As a consequence of Proposition \ref{prop:EmbeddingExTree}, $C(T_+)$ is isomorphic to a subgroup of $C(T)$. Notice that the subgroup $\langle x \cup V(T) \rangle$ of $C(T_+)$ is isomorphic to $(\mathbb{Z}/2\mathbb{Z}) \ast C(T)$, which contains $C(T) \ast C(T)$. By applying this observation $k$ times, we get the desired conclusion.
\end{proof}

\noindent
In the next statement, we denote by $\Xi$ the \emph{double-star}, i.e., the smallest tree containing two adjacent vertices of degree $3$. 

\begin{cor}\label{cor:EmbeddingExTree}
For any finite forest $F$, $C(F)$ is isomorphic to a subgroup of $C(\Xi)$. 
\end{cor}

\begin{proof}
According to Corollary \ref{cor:ForestFromTree}, it suffices to show that, for every finite tree $T$, $C(T)$ embeds into $C(\Xi)$. Let $a,b \in \Xi$ the two adjacent vertices of degree $3$. Fixing a basepoint $x \in T$, notice that
$$\left\{ \begin{array}{ccc} T & \to & \Xi \\ u & \mapsto & \left\{ \begin{array}{cl} a & \text{if $d_T(x,u)$ is even} \\ b & \text{if $d_T(x,u)$ is odd} \end{array} \right. \end{array} \right.$$
defines a graph morphism. We conclude from Proposition \ref{prop:EmbeddingExTree} that $C(T)$ embeds into~$C(\Xi)$. 
\end{proof}

\noindent
In the next statement, for every $n \geq 0$ we denote by $S_n$ the path of length $n$ (which contains $n+1$ vertices).

\begin{cor}\label{cor:EmbeddingExSegment}
For every $n \geq 0$, $C(S_n)$ is isomorphic to a subgroup of $C(S_3)$.
\end{cor}

\begin{proof}
If $n \leq 3$, then $S_n$ is an induced subgraph of $S_3$, so there is nothing to prove. From now on, assume that $n \geq 4$. Let $a,b,c,d$ denote the consecutive vertices of $S_3$ and $x_1, \ldots, x_{n+1}$ those of $S_n$. If $n=2p$ for some $p \geq 2$, then
$$\left\{ \begin{array}{ccc} S_n & \to & S_3 \\ x_i & \mapsto & \left\{ \begin{array}{cl} a & \text{if $i=1$ or $i=n+1$} \\ b & \text{if $i$ is even} \\ c & \text{if $1<i<n+1$ is odd} \end{array} \right. \end{array} \right.$$
defines a graph morphism; and if $n=2p-1$ for some $p \geq 3$, then
$$\left\{ \begin{array}{ccc} S_n & \to & S_3 \\ x_i & \mapsto & \left\{ \begin{array}{cl} a & \text{if $i=1$} \\ b & \text{if $i<n+1$ is even} \\ c & \text{if $1<i$ is odd} \\ d & \text{if $i=n+1$} \end{array} \right. \end{array} \right.$$
defines a graph morphism as well. The desired conclusion follows from Proposition~\ref{prop:EmbeddingExTree}. 
\end{proof}

\noindent
We conclude this section by illustrating Proposition \ref{prop:EmbeddingExTree} with an eplicit example.

\begin{ex}\label{ex:GraphMorphism}
Let $R$ and $S$ be the two finite trees given by Figure \ref{GraphMorphism}. The map $R \to S$ provided by the same figure satisfies the assumptions of Proposition \ref{prop:EmbeddingExTree}, so $C(R)$ contains a subgroup isomorphic to $C(S)$. 
\begin{figure}
\begin{center}
\includegraphics[scale=0.5]{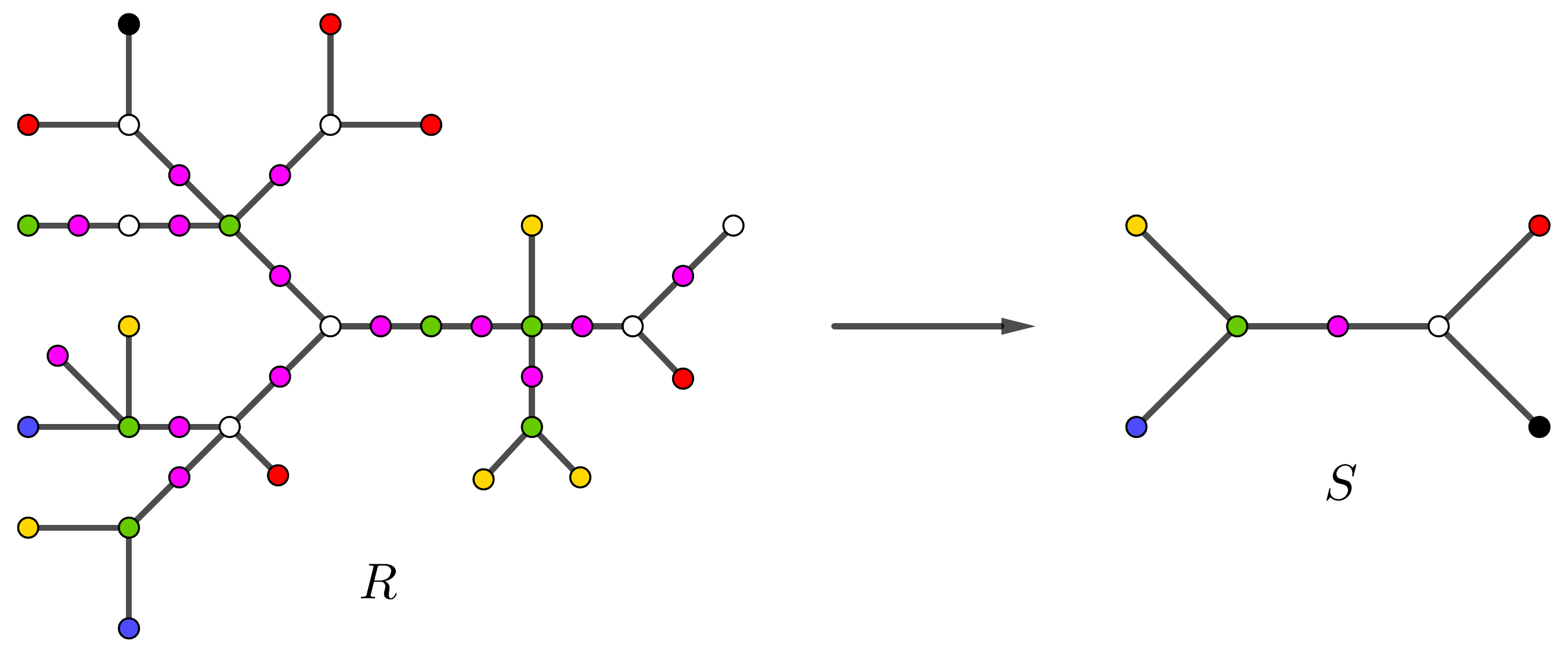}
\caption{A graph morphism $R \to S$ satisfying the assumptions of Proposition \ref{prop:EmbeddingExTree}.}
\label{GraphMorphism}
\end{center}
\end{figure}
\end{ex}

\subsection{Constructing finite-index subgroups.}\label{section:FI}

\noindent
Theorem \ref{thm:ReflectionGeneral} is particularly useful to construct finite-index subgroups in right-angled Coxeter groups. For instance, if $\Gamma$ is a finite simplicial graph and $Y$ a finite convex subcomplex of $X(\Gamma)$, then it follows from Theorem \ref{thm:ReflectionGeneral} that the subgroup $H \leq C(\Gamma)$ generated by the reflections along hyperplanes tangent to $Y$ has $F$ as a fundamental domain; in particular, $H$ has finite index in $C(\Gamma)$. As an easy application, it can be proved that $C(\Gamma)$ is residually finite. More generally, we are able to give a short proof of \cite[Theorem A]{MR2413337}, which states that convex-cocompact subgroups of $C(\Gamma)$ are \emph{separable}. (Recall that, given a group $G$, a subgroup $H \leq G$ is \emph{separable} if, for every $g \notin H$, there exists a finite-index subgroup $K \leq G$ which contains $H$ but not $g$.)

\begin{thm}\emph{\cite{MR2413337}}
Let $\Gamma$ be a finite simplicial graph. Convex-cocompact subgroups of $C(\Gamma)$ are virtual retracts.
\end{thm}

\begin{proof}
Let $H \leq C(\Gamma)$ be a subgroup acting cocompactly on a convex subcomplex $Y \subset X(\Gamma)$. Without loss of generality, suppose that $Y$ contains the vertex $1$. Let $\mathcal{J}$ be the collection of all the hyperplanes tangent to $Y$ and let $R$ denote the subgroup of $C(\Gamma)$ generated by the reflections along hyperplanes of $\mathcal{J}$. Set $H^+= \langle H, R \rangle$. Notice that $\mathcal{J}$ is peripheral, so that Theorem \ref{thm:ReflectionGeneral} applies to $R$ and shows that $Y$ is a fundamental domain of $R \curvearrowright X(\Gamma)$. As $H$ acts cocompactly on $Y$, it follows that $H^+$ acts cocompactly on $X(\Gamma)$. Therefore, $H^+$ is a finite-index subgroup of $C(\Gamma)$. 

\medskip \noindent
Next, by applying Theorem \ref{thm:ReflectionGeneral} to $H^+$ and $\mathcal{J}$, it follows that $H^+$ decomposes as $R \rtimes \mathrm{stab}_{H^+}(Y)$. Notice that $H \subset \mathrm{stab}_{H^+}(Y)$. Conversely, let $h$ be an element of $\mathrm{stab}_{H^+}(Y)$. As $R$ and $H$ generate $H^+$, we can write $h=h_1 r_1 \cdots h_kr_k$ for some $r_1,\ldots, r_k \in R$ and $h_1, \ldots, h_k \in H$. Then
$$h = \left( h_1r_1h_1^{-1} \right) \cdot \left( h_1h_2r_2h_2^{-1}h_1^{-1} \right) \cdots \left( h_1 \cdots h_k r_k h_k^{-1} \cdots h_1^{-1} \right) \cdot h_1 \cdots h_k,$$
hence $h \in R \cdot h_1 \cdots h_k$ because $R$ is a normal subgroup in $H^+$. But $h$ belongs to $\mathrm{stab}_{H^+}(Y)$ in the semidirect product $R \rtimes \mathrm{stab}_{H^+}(Y)$, so we must have $h = h_1 \cdots h_k \in H$. 

\medskip \noindent
Thus, we have proved that the finite-index subgroup $H^+$ decomposes as a semidirect product $R \rtimes H$. The projection onto the second factor $R \rtimes H \to H$ provides a retraction $H^+ \to H$, as desired.
\end{proof}

\noindent
Because virtual retracts are automatically separable, it follows that:

\begin{cor}
Let $\Gamma$ be a finite simplicial graph. Convex-cocompact subgroups of $C(\Gamma)$ are separable.
\end{cor}

\section{Morphisms between right-angled Coxeter groups}

\subsection{Reflections and morphisms}\label{section:ReMo}

\noindent
In order to state the main result of this section, we need to introduce the following definition:

\begin{definition}
Let $\Phi,\Psi$ be two simplicial graphs. A morphism $\rho : C(\Phi) \to C(\Psi)$ is a \emph{folding} if there exists an equivalence class $\sim$ on the vertices of $\Phi$ such that $\Psi = \Phi / \sim$ and $\rho$ sends each generator of $\Phi$ to the generator of $\Psi$ corresponding to the equivalence class of the associated vertex of $\Phi$.
\end{definition}

\noindent
The main step towards the description of morphisms between two-dimensional right-angled Coxeter groups is following theorem, which is the main statement of the section. We emphasize that, in this statement, the graph $\Gamma_2$ is not assumed triangle-free.

\begin{thm}\label{thm:MorphismReflection}
Let $\Gamma_1,\Gamma_2$ be two finite simplicial graphs and $\rho : C(\Gamma_1) \to C(\Gamma_2)$ a morphism. Assume that $\rho(u)$ is a reflection for every $u \in V(\Gamma_1)$. There exist 
\begin{itemize}
	\item a sequence of graphs $\Lambda_0=\Gamma_1, \Lambda_1, \ldots, \Lambda_k$;
	\item partial conjugations $\alpha_1 \in \mathrm{Aut}(C(\Lambda_1)), \ldots, \alpha_k \in \mathrm{Aut}(C(\Lambda_k))$;
	\item foldings $\pi_1 : C(\Lambda_0) \to C(\Lambda_1), \ldots, \pi_k : C(\Lambda_{k-1}) \to C(\Lambda_k)$;
	\item and a peripheral embedding $\bar{\rho} : C(\Lambda_k) \hookrightarrow C(\Gamma_2)$
\end{itemize}
such that $\rho = \bar{\rho} \circ \pi_k \circ \alpha_{k-1} \circ \pi_{k-1} \circ \cdots \circ \alpha_1 \circ \pi_1$.
\end{thm}

\noindent
Let us recall the definition of partial conjugations used in the previous statement:

\begin{definition}
Let $\Gamma$ be a simplicial graph, $u \in V(\Gamma)$ a vertex, and $\Lambda$ a connected component of $\Gamma \backslash \mathrm{star}(u)$ (i.e., of the complement of $u$ and all its neighbors). The map
$$w \mapsto \left\{ \begin{array}{cl} w & \text{if $w \notin V(\Lambda)$} \\ uwu^{-1} & \text{if $w \in V(\Lambda)$} \end{array} \right.$$
induces an automorphism of $C(\Gamma)$, referred to as a \emph{partial conjugation}. 
\end{definition}

\noindent
Before turning to the proof of Theorem \ref{thm:MorphismReflection}, we emphasize that the alternating sequence of foldings and partial conjugations in the statement cannot be replaced with the concatenation of a well-chosen folding and a well-chosen partial conjugation. We refer to Example \ref{ex:alternating} for a justification. 

\medskip \noindent
For the purpose of the proof of Theorem \ref{thm:MorphismReflection}, we introduce the following definition:

\begin{definition}
Let $\Phi,\Psi$ be two finite simplicial graphs. For any morphism $\varphi : C(\Phi) \to C(\Psi)$ sending each generator $u$ of $C(\Phi)$ to a reflection along a hyperplane $J_u$ of $X(\Psi)$, the \emph{complexity} of $\varphi$ is $\displaystyle \chi(\varphi):= \sum\limits_{u \in V(\Phi)} d(1,N(J_u))$. 
\end{definition}

\medskip \noindent
Combined with our next lemma, the complexity will be used in order to argue by induction. 

\begin{lemma}\label{lem:InductionComplexity}
Let $\Phi,\Psi$ be two finite simplicial graphs and $\varphi : C(\Phi) \to C(\Psi)$ a morphism sending each generator $u$ of $C(\Phi)$ to a reflection $r_u$ along a hyperplane $J_u$ of $X(\Psi)$. If $\mathcal{J}:= \{ J_u \mid u \in V(\Phi)\}$ is not peripheral, then there exists a partial conjugation $\alpha \in \mathrm{Aut}(C(\Phi))$ such that $\chi(\varphi \alpha)< \chi(\varphi)$. 
\end{lemma}

\noindent
This lemma is essentially contained in the proof of \cite[Lemma 4.12]{AutGP}. We write a complete proof for the reader's convenience.

\begin{proof}[Proof of Lemma \ref{lem:InductionComplexity}.]
Because $\mathcal{J}$ is not peripheral, there exist two distinct vertices $a,b \in V(\Phi)$ such that $J_a$ separates $1$ from $J_b$. Notice that $a$ and $b$ are not adjacent in $\Phi$ as the hyperplanes $J_a$ and $J_b$ are not transverse. Now, let $\alpha$ denote the partial conjugation of $C(\Phi)$ which conjugates by $a$ the generators of the connected component $\Lambda$ of $\Phi \backslash \mathrm{star}(a)$ which contains $b$. We claim that $\chi( \varphi \alpha ) < \chi ( \varphi )$. 

\medskip \noindent
Notice that, for every $u \notin \Lambda$, $\varphi \alpha (u)= \varphi (u)$ is the reflection $r_u$ along $J_u$; and that, for every $u \in \Lambda$, $\varphi \alpha (u) = \varphi ( aua^{-1}) = r_a r_u r_a^{-1}$ is a reflection along $r_aJ_u$. Therefore, in order to prove the inequality $\chi( \varphi \alpha ) < \chi( \varphi )$, it is sufficient to show that $d(1,N(r_a J_u))< d(1,N(J_u))$ for every $u \in \Lambda$. It is a consequence of Lemma \ref{lem:shortendist} and of the following observation:

\begin{fact}
For every $u \in \Lambda$, $J_a$ separates $1$ from $J_u$.
\end{fact}

\noindent
Let $u \in \Lambda$ be a vertex. By definition of $\Lambda$, there exists a path 
$$u_1=u, \ u_2, \ldots, u_{n-1}, \ u_n=b$$
in $\Phi$ that is disjoint from $\mathrm{star}(a)$. This yields a path
$$J_{u_1}=J_u, \ J_{u_2}, \ldots, J_{u_{n-1}}, \ J_{u_n}=J_b$$
in the crossing graph of $X(\Gamma)$. As a consequence of Lemma \ref{lem:transverseimpliesadj}, none of these hyperplanes are transverse to $J_a$, which implies that they are all contained in the same halfspace delimited by $J_a$, namely the one containing $J_b$. As $J_a$ separates $1$ from $J_b$, this concludes the proof of our fact, and finally to our lemma. 
\end{proof}

\begin{proof}[Proof of Theorem \ref{thm:MorphismReflection}.]
Let $\Lambda_1$ denote the graph obtained from $\Gamma_1$ by identifying any two vertices having the same image under $\rho$. Notice that $\rho$ factors through $C(\Lambda_1)$, i.e., there exists a morphism $\rho_1 : C(\Lambda_1) \to C(\Gamma_2)$ such that $\rho = \rho_1 \circ \pi_1$ where $\pi_1$ is the folding $C(\Gamma_1) \to C(\Lambda_1)$. By construction, $\rho_1$ sends the generators of $C(\Lambda_1)$ to pairwise distinct reflections. Next, if $\rho_1$ sends the generators of $C(\Lambda_1)$ to reflections along a peripheral collection of hyperplanes, then we are done. If it is not the case, then we apply Lemma~\ref{lem:InductionComplexity} and find a partial conjugation $\beta_1 \in \mathrm{Aut}(C(\Lambda_1))$ such that $\chi(\rho_1 \circ \beta_1)< \chi(\rho_1)$. Notice that $\rho = (\rho \circ \beta_1) \circ \beta_1^{-1} \circ \pi_1$. 

\medskip \noindent
The desired conclusion follows by induction on the complexity.
\end{proof}

\subsection{Diagonal embeddings}\label{section:Diagonal}

\noindent
In this section, our goal is to describe the morphisms which exist between two right-angled Coxeter groups of arbitrary dimensions. Before stating our main result in this direction, we need the following definition:

\begin{definition}
Let $\Phi$ and $\Psi$ be two simplicial graph. A morphism $\rho : C(\Phi) \to C(\Psi)$ is a \emph{diagonal} if it sends each generator of $C(\Phi)$ to a product of pairwise commuting generators of $C(\Psi)$. 
\end{definition}

\noindent
Our goal is to prove the following statement:

\begin{thm}\label{thm:MorphismAnyDim}
Let $\Phi, \Psi$ be two finite graphs and $\rho : C(\Phi) \to C(\Psi)$ a morphism. Then there exist 
\begin{itemize}
	\item a sequence of graphs $\Lambda_1, \ldots, \Lambda_p$;
	\item a diagonal morphism $\delta : C(\Phi) \hookrightarrow C(\Lambda_1)$;
	\item partial conjugations $\alpha_1 \in \mathrm{Aut}(C(\Lambda_1)), \ldots, \alpha_{p-1} \in \mathrm{Aut}(C(\Lambda_{p-1}))$;
	\item foldings $\pi_1 : C(\Lambda_1) \to C(\Lambda_2), \ldots, \pi_{p-1} : C(\Lambda_{p-1}) \to C(\Lambda_p)$;
	\item and a peripheral embedding $\bar{\rho} : C(\Lambda_p) \hookrightarrow C(\Psi)$
\end{itemize}
such that $\rho = \bar{\rho} \circ \pi_{p-1} \circ \alpha_{p-1} \circ \cdots \circ \pi_1 \circ \alpha_1 \circ \delta$. 
\end{thm}

\begin{proof}
For every generator $u \in C(\Phi)$, $\rho(u)$ must have order two in $C(\Psi)$ so it can be written as $gs_1 \cdots s_k g^{-1}$ where $g \in C(\Psi)$ and where $s_1, \ldots, s_k$ are pairwise commuting generators. In other words, $\rho(u)$ is the product of the reflections $gs_1g^{-1}, \ldots, gs_kg^{-1}$ along the hyperplanes $gJ_{s_1},\ldots, gJ_{s_k}$. Let $\mathcal{J}$ be the collection of all these hyperplanes where $u$ runs over the generators of $C(\Phi)$. If $\Lambda_1$ denotes the crossing graph of $\mathcal{J}$, then the morphism $\rho : C(\Phi) \to C(\Psi)$ factors through $C(\Lambda_1)$, i.e., there exist two morphisms $\delta : C(\Phi) \to C(\Lambda_1)$ and $\varphi : C(\Lambda_1) \to C(\Psi)$ such that $\rho= \varphi \circ \delta$, such that $\delta$ sends each generator of $C(\Phi)$ to a product of pairwise commuting generators of $C(\Lambda_1)$, and such that $\varphi$ sends each generator of $C(\Lambda_1)$ to a reflection of $C(\Psi)$. Applying Theorem \ref{thm:MorphismReflection} to $\varphi : C(\Lambda_1) \to C(\Psi)$ yields the desired conclusion. 
\end{proof}

\subsection{Morphisms in dimension two}\label{section:MainThm}

\noindent
When restricting ourselves to two-dimensional right-angled Coxeter groups, it is possible to simplify Theorem \ref{thm:MorphismAnyDim} by replacing the diagonal morphism with a composition of even simpler morphisms. More precisely, the main statement of this section is:

\begin{thm}\label{thm:Morphisms}
Let $\Phi,\Psi$ be two finite simplicial graphs and $\rho : C(\Phi) \to C(\Psi)$ a morphism. Assume that $\Psi$ is triangle-free. There exist a sequence of graphs $\Lambda_0= \Phi, \Lambda_1, \ldots, \Lambda_p$ and morphisms $\bar{\rho} : C(\Lambda_p) \to C(\Psi)$ and $\alpha_i : C(\Lambda_i) \to C(\Lambda_{i+1})$ for $0 \leq i \leq p-1$ such that
\begin{itemize}
	\item for every $0 \leq i \leq p-1$, $\alpha_i$ is an automorphism, a folding, an ingestion or an erasing;
	\item $\Lambda_p$ decomposes as the disjoint union of $q$ isolated vertices with a graph $\Gamma$ without isolated vertex, and $\bar{\rho}$ is a free product $\psi_1 \ast \cdots \ast \psi_q \ast \psi$ where $\psi : C(\Gamma) \to C(\Psi)$ is a peripheral embedding and $\psi_i : \mathbb{Z}/2\mathbb{Z} \to C(\Psi)$ for $1 \leq i \leq q$;
	\item the following diagram is commutative:
\[
\xymatrix{C(\Phi) = C(\Lambda_0) \ar[rr]^\rho \ar[d]_{\alpha_0} & & C(\Psi) \\ C(\Lambda_1) \ar[d]_{\alpha_1} & & \\ \vdots \ar[d]_{\alpha_{p-1}} & & \\ C(\Lambda_p)= (\mathbb{Z}/2 \mathbb{Z}) \ast \cdots \ast (\mathbb{Z}/2\mathbb{Z}) \ast C(\Gamma) \ar[rruuu]_{\bar{\rho}= \psi_1 \ast \cdots \ast \psi_q \ast \psi} & & }
\]
\end{itemize}
\end{thm}

\noindent
Before turning to the proof, we need to define some of the terms used in the previous statement.

\begin{definition}
Let $\Phi,\Psi$ be two simplicial graphs. A morphism $\rho : C(\Phi) \to C(\Psi)$ is an \emph{erasing} if $\Psi$ is a subgraph of $\Phi$ and $\rho$ fixes each generator corresponding to a vertex of $\Psi$ and sends to $1$ all the other generators.
\end{definition}

\begin{definition}
Let $\Phi,\Psi$ be two simplicial graphs. A morphism $\rho : C(\Phi) \to C(\Psi)$ is an \emph{ingestion} if there exists a vertex $a \in V(\Phi)$ having exactly two neighbors $b,c \in V(\Phi)$ such that $\Psi = \Gamma \backslash \{a\}$ and such that $\rho$ fixes each generator corresponding to a vertex of $\Psi$ and sends $a$ to $bc$. 
\end{definition}

\noindent
Theorem \ref{thm:Morphisms} will be an easy consequence of the combination of Theorem \ref{thm:MorphismReflection} from the previous section with the following proposition:

\begin{prop}\label{prop:Morphisms}
Let $\Phi,\Psi$ be two finite simplicial graphs and $\rho : C(\Phi) \to C(\Psi)$ a morphism. Assume that $\Psi$ is triangle-free. There exist an erasing $\alpha : C(\Phi) \to C(\Phi')$, a folding $\beta : C(\Phi') \to C(\Phi'')$, a product of ingestions and transvections $\gamma : C(\Phi'') \to C(\Phi''')$, and a morphism $\bar{\rho} : C(\Phi''') \to C(\Psi)$ such that the diagram 
\[
\xymatrix{C(\Phi) \ar[rr]^\rho \ar[d]_\alpha & & C(\Psi)  \\ C(\Phi') \ar[d]_\beta & & \\ C(\Phi'') \ar[d]_\gamma & & \\ C(\Phi''') \ar[rruuu]_{\bar{\rho}} & & }
\]
is commutative, and such that $\bar{\rho}$ sends each generator of $C(\Phi''')$ which does not correspond to an isolated vertex of $\Phi'''$ to a reflection of $C(\Psi)$. 
\end{prop}

\noindent
Let us recall the definition of transvections, used in the previous statement:

\begin{definition}
Let $\Gamma$ be a simplicial graph and $u,v \in V(\Gamma)$ two adjacent vertices such that any neighbor of $u$ distinct from $v$ is also a neighbor of $v$. The map $$\left\{ \begin{array}{ccc} V(\Gamma) & \to & V(\Gamma) \\ w & \mapsto & \left\{ \begin{array}{cl} w & \text{if $w \neq u$} \\ wv & \text{if $w=u$} \end{array} \right. \end{array} \right.$$ induces an automorphism $C(\Gamma) \to C(\Gamma)$, referred to as a \emph{transvection}. 
\end{definition}

\noindent
Before turning to the proof of Proposition \ref{prop:Morphisms}, we begin by proving a preliminary lemma.

\begin{lemma}\label{lem:Commute}
Let $G$ be a group acting on a two-dimensional CAT(0) cube complex $X$ with trivial vertex-stabilisers. Fix two distinct elements $a,b \in G$, and assume that $a$ is the product of two reflections along two transverse hyperplanes $J, H$ and that $b$ is either a reflection or the product of two reflections along two transverse hyperplanes. If $a$ and $b$ commute, then $b$ is a reflection along either $J$ or $H$. 
\end{lemma}

\begin{proof}
We begin by proving a few elementary observations.

\begin{fact}\label{fact:Unique}
If $r_1,r_2$ are two reflections of $G$ along the same hyperplane $J$, then $r_1=r_2$.
\end{fact}

\noindent
Indeed, $r_1r_2^{-1}$ fixes pointwise each edge dual to $J$, hence $r_1=r_2$ as $G$ acts with trivial vertex-stabilisers.

\begin{fact}\label{fact:Square}
Let $r_A,r_B$ be two reflections respectively along two transverse hyperplanes $A,B$, and let $S$ be a square crossed by both $A$ and $B$. Then $S$ is the unique square stabilised by $r_Ar_B$.
\end{fact}

\noindent
It is clear that $S$ is stabilised by $r_Ar_B$. Moreover, if a square $S'$ is not crossed by $A$ then $r_AS'$ and $S'$ are separated by $A$, and because $r_B$ stabilises each halfspace delimited by $A$ as a consequence of Claim \ref{claim:HalfStab}, $r_Br_AS'$ must be distinct from $S'$. Because $r_A$ and $r_B$ commute as a consequence of Claim \ref{Claim:TransverseCommute}, we conclude that $r_Ar_BS' \neq S'$. Similarly, if $S'$ is a square which is not crossed by $B$, then $r_Ar_BS' \neq S'$. Thus, we have proved that any square left invariant by $r_Ar_B$ must be included into $N(A) \cap N(B)$. The desired conclusion follows from the observation that, as $N(A) \cap N(B)$ cannot be crossed by a hyperplane distinct from $A$ and $B$ (otherwise, $X$ would contain three pairwise transverse hyperplane, contradicting the fact that $X$ is two-dimensional), we must have $N(A) \cap N(B)=S$.

\medskip \noindent
Now let us go back to the proof of our lemma. Let $S$ denote the unique square crossed by both $J$ and $H$. Because $a$ and $b$ commute and that $a$ stabilises $S$, it follows that $a$ also stabilises $bS$, hence $bS=S$ as a consequence of the previous fact. Consequently, if $b$ is the product of two reflections along two transverse hyperplanes, then it follows from Fact \ref{fact:Square} that these hyperplanes must be $J,H$. Then Fact \ref{fact:Unique} implies that $b=a$, contradicting our assumption. Therefore, $b$ has to be a reflection along some hyperplane, but since we also know that $b$ stabilises $S$, this hyperplane must be $J$ or $H$, concluding the proof of our lemma. 
\end{proof}

\begin{proof}[Proof of Proposition \ref{prop:Morphisms}.]
Let $\Phi'$ be the graph obtained from $\Phi$ by removing all the vertices whose associated generators of $C(\Phi)$ are sent to $1$ by $\rho$. Then there exist a morphism $\rho' : C(\Phi') \to C(\Psi)$ and an erasing $\alpha : C(\Phi) \to C(\Phi')$ such that $\rho = \rho' \circ \alpha$ and such that $\rho'$ does not send any generator of $C(\Phi')$ to $1$. Next, let $\Phi''$ be the graph obtained from $\Phi'$ by identifying any two vertices whose associated generators of $C(\Phi')$ have the same image under $\rho'$. Then there exist a morphism $\rho'' : C(\Phi'') \to C(\Psi)$ and a folding $\beta : C(\Phi') \to C(\Phi'')$ such that $\rho' = \rho'' \circ \beta$ and such that the generators of $C(\Phi'')$ are sent by $\rho''$ to pairwise distinct non-trivial elements of $C(\Psi)$. 

\medskip \noindent
Because $\Psi$ is triangle-free, for every $u \in V(\Phi'')$ two cases may happen: either $\rho''(u)$ is a reflection (i.e., $\rho''(u)=gvg^{-1}$ for some $v \in V(\Psi)$ and $g \in C(\Psi)$), or $\rho''(u)$ is the product of two reflections along two transverse hyperplanes (i.e., $\rho''(u)=gvwg^{-1}$ for some $v,w \in V(\Psi)$ adjacent and $g \in C(\Psi)$). Call a vertex of $\Phi''$ \emph{bad} if it is not isolated and if its image under $\rho''$ is not a reflection. If $\Phi''$ does not contain any bad vertex, we are done. From now on, assume that there exists at least one bad vertex $u \in V(\Phi'')$. Let $v,w \in V(\Psi)$ and $g \in C(\Psi)$ be such that $\rho''(u)=gvwg^{-1}$. 

\medskip \noindent
Notice that, as a consequence of Lemma \ref{lem:Commute}, if $u$ has at least three neighbors in $\Phi''$, then two of them must be sent by $\rho''$ to the same reflection (along $gJ_v$ or $gJ_w$), which would contradicts the fact that $\rho''$ is injective on the generators. Consequently, $u$ has either one or two neighbors.

\medskip \noindent
Assume first that $u$ has exactly one neighbor. Let $z \in \Phi''$ denote this neighbor and let $\gamma \in \mathrm{Aut}(C(\Phi'')$ be the transvection satisfying $\gamma(u)=uz$. It follows from Lemma \ref{lem:Commute} that $\rho''(z)$ is equal to either $gvg^{-1}$ or $gwg^{-1}$. Up to swapping $v$ and $w$, we assume that we are in the former case. Notice that $\rho'' \circ \gamma$ agrees with $\rho''$ for all the generators of $C(\Phi'')$ different from $u$ and that $$\rho'' \circ \gamma(u) = \rho''(u) \cdot \rho''(z) = gvwg^{-1} \cdot gvg^{-1}= gwg^{-1}.$$
Therefore, we have $\rho'' = (\rho'' \circ \gamma) \circ \gamma$ where $\gamma$ is a transvection and where $\rho'' \circ \gamma$ has less bad vertices than $\rho''$. 

\medskip \noindent
Now, assume that $u$ has exactly two neighbors. Let $x,y \in \Phi''$ denote these neighbors. As a consequence of Lemma \ref{lem:Commute}, we have $\rho''(x)=gvg^{-1}$ and $\rho''(y)=gwg^{-1}$ (up to swapping $v$ and $w$). It follows that $x$ and $y$ must be adjacent and that $\rho''(u)=\rho''(xy)$. Therefore, if $\gamma : C(\Phi'') \to C(\Phi'' \backslash \{u\})$ denotes the ingestion sending $u$ to $xy$, there exists a morphism $\rho''' : C(\Phi'' \backslash \{u\}) \to C(\Psi)$ such that $\rho'' = \rho''' \circ \gamma$. Notice that $\rho'''$ and $\rho''$ agree on $\Gamma \backslash \{u\}$, so that $\rho'''$ has less bad vertices than $\rho''$.

\medskip \noindent
By induction on the number of bad vertices of $\rho''$, we conclude that there exist a product of transvections and ingestions $\gamma : C(\Phi'') \to C(\Phi''')$ and a morphism $\rho''' : C(\Phi''') \to C(\Psi)$ such that $\rho'' = \rho''' \circ \gamma$ and such that $\rho'''$ has no bad vertex, i.e., every generator of $C(\Phi''')$ which does not correspond to an isolated vertex is sent by $\rho'''$ to a reflection. 
\end{proof}

\begin{proof}[Proof of Theorem \ref{thm:Morphisms}.]
First, we apply Proposition \ref{prop:Morphisms} and we find an erasing $\alpha : C(\Phi) \to C(\Phi')$, a folding $\beta : C(\Phi') \to C(\Phi'')$, a product of ingestions and transvections $\gamma : C(\Phi'') \to C(\Phi''')$, and a morphism $\bar{\rho} : C(\Phi''') \to C(\Psi)$ such that $\rho = \bar{\rho} \circ \gamma \circ \beta \circ \alpha$ and such that $\bar{\rho}$ sends each generator of $C(\Phi''')$ which does not correspond to an isolated vertex of $\Phi'''$ to a reflection of $C(\Psi)$. Decompose $\Phi'''$ as the disjoint union of $q$ isolated vertices and a graph $\Xi$ which does not contain any isolated vertex. As $C(\Phi''')$ decomposes as the free product $(\mathbb{Z}/2\mathbb{Z}) \ast \cdots \ast (\mathbb{Z} /2 \mathbb{Z}) \ast C(\Xi)$, the morphism $\bar{\rho}$ decomposes as a free product $\psi_1 \ast \cdots \ast \psi_q \ast \varphi$ where $\varphi : C(\Xi) \to C(\Psi)$ and $\psi_i : \mathbb{Z}/2 \mathbb{Z} \to C(\Psi)$ for $1 \leq i \leq q$. 

\medskip \noindent
Now, we apply Theorem \ref{thm:MorphismReflection} to $\varphi$ and we find
\begin{itemize}
	\item a sequence of graphs $\Lambda_0=\Xi, \Lambda_1, \ldots, \Lambda_k$;
	\item partial conjugations $\alpha_1 \in \mathrm{Aut}(C(\Lambda_1)), \ldots, \alpha_k \in \mathrm{Aut}(C(\Lambda_k))$;
	\item foldings $\pi_1 : C(\Lambda_0) \to C(\Lambda_1), \ldots, \pi_k : C(\Lambda_{k-1}) \to C(\Lambda_k)$;
	\item and a peripheral embedding $\bar{\varphi} : C(\Lambda_k) \hookrightarrow C(\Psi)$
\end{itemize}
such that $\varphi = \bar{\varphi} \circ \pi_k \circ \alpha_{k-1} \circ \pi_{k-1} \circ \cdots \circ \alpha_1 \circ \pi_1$.

\medskip \noindent
For every $0 \leq i \leq k$, let $\Lambda_i'$ denote the disjoint union of $\Lambda_i$ with $q$ isolated vertices. Notice that $\Lambda_0$ can be identified to $\Phi''$. It is worth noticing that, for every $1 \leq i \leq k$, $\alpha_i$ naturally defines a partial conjugation of $C(\Lambda_i')$. More precisely, there exists a partial conjugation $\alpha_i' \in \mathrm{Aut}(C(\Lambda_i'))$ which agrees with $\alpha_i$ on $\Lambda_i$ and which fixes the other generators. Similarly, there exists a folding $\pi_i' : C(\Lambda_{i-1}') \to C(\Lambda_i')$ which agrees with $\pi_i$ on $\Lambda_{i-1}$ and which fixes the other generators. 

\medskip \noindent
We finally have
$$\begin{array}{lcl} \rho & = & \bar{\rho} \circ \gamma \circ \beta \circ \alpha = (\psi_1 \ast \cdots \ast \psi_q \ast \varphi) \circ \gamma \circ \beta \circ \alpha \\ \\ & = &  \left( \psi_1 \ast \cdots \ast \psi_q \ast \left( \bar{\varphi} \circ \pi_k \circ \alpha_{k-1} \circ \pi_{k-1} \circ \cdots \circ \alpha_1 \circ \pi_1 \right) \right) \circ \gamma \circ \beta \circ \alpha \\ \\ & = & (\psi_1 \ast \cdots \ast \psi_q \ast \bar{\varphi}) \circ \pi_k' \circ \alpha_{k-1}' \circ \pi_{k-1}' \circ \cdots \circ \alpha_1' \circ \pi_1' \circ \gamma \circ \beta \circ \alpha \end{array}$$
which is the desired decomposition of $\rho$. 
\end{proof}

\noindent
Because foldings, ingestions and erasing have non-trivial kernels, an immediate consequence of Theorem \ref{thm:Morphisms} is:

\begin{cor}\label{cor:Embedding}
Let $\Phi,\Psi$ be two finite simplicial graphs and $\rho : C(\Phi) \hookrightarrow C(\Psi)$ an injective morphism. Assume that $\Psi$ is triangle-free and that $\Phi$ has no isolated vertex. There exists an automorphism $\alpha \in \mathrm{Aut}(C(\Phi))$ such that $\rho \circ \alpha : C(\Phi) \hookrightarrow C(\Psi)$ is a peripheral embedding. 
\end{cor}\qed

\noindent
Notice that Corollary \ref{cor:Embedding} generalises \cite[Theorem C]{RACGStallings}. As a straightforward application, we deduce the following solution to the embedding problem among two-dimensional right-angled Coxeter groups:

\begin{cor}\label{cor:Sub}
Let $\Phi,\Psi$ be two finite simplicial graphs. Assume that $\Psi$ is triangle-free. Then $C(\Phi)$ is isomorphic to a subgroup of $C(\Psi)$ if and only if $X(\Psi)$ contains a peripheral collection of hyperplanes whose crossing graph is isomorphic to $\Phi$. 
\end{cor}

\begin{proof}
Assume that $C(\Phi)$ is isomorphic to a subgroup of $C(\Psi)$. Decompose $\Phi$ as the disjoint union of $n$ isolated vertices with a graph $\Gamma$ which does not contain any isolated vertex. So $C(\Phi)$ decomposes as the free product of $n$ copies of $\mathbb{Z}/2\mathbb{Z}$ with $C(\Gamma)$. If $n=0$, then the desired conclusion follows from Corollary \ref{cor:Embedding}, so we suppose that $n \geq 1$. We distinguish two cases.

\medskip \noindent
First, assume that $\Gamma$ is non-empty, so that $C(\Gamma)$ contains a subgroup isomorphic to $(\mathbb{Z}/2\mathbb{Z}) \oplus (\mathbb{Z}/2 \mathbb{Z})$. Let $\Phi'$ denote the disjoint union of $n$ isolated edges with $\Gamma$. By noticing that $(\mathbb{Z}/2\mathbb{Z}) \ast C(\Gamma)$ contains a subgroup isomorphic to the free product of $n+1$ copies of $C(\Gamma)$, it follows that $C(\Phi')$ is isomorphic to a subgroup of $C(\Phi)$, and so of $C(\Psi)$. Because $\Phi'$ has no isolated vertex, we deduce from Corollary \ref{cor:Embedding} that $X(\Psi)$ contains a peripheral collection of hyperplanes whose crossing graph is $\Phi'$. The desired conclusion follows by extracting the subcollection corresponding to the subgraph $\Phi$ of~$\Phi'$. 

\medskip \noindent
Next, assume that $\Gamma$ is empty, so that $C(\Phi)$ is isomorphic to a free product of $n$ copies of $\mathbb{Z}/2\mathbb{Z}$. If $n=1$, just pick a single hyperplane of $X(\Psi)$ (such a hyperplane exists because the fact that $C(\Phi)$ embeds into $C(\Psi)$ implies that $C(\Psi)$ is non-trivial, so $\Psi$ cannot be empty); and if $n=2$, just pick two non-transverse hyperplanes of $X(\Psi)$ (such hyperplanes exist because the fact that $C(\Phi)$ embeds into $C(\Psi)$ implies that $C(\Psi)$ is infinite, so $\Psi$ has to contain two non-adjacent vertices). From now on, assume that $n \geq 3$. Decompose $\Psi$ as the join of a complete graph $\Psi_0$ with subgraphs $\Psi_1, \ldots, \Psi_k$ such that each $\Psi_i$ does not decompose as a join and is not reduced to a single vertex. Notice that $C(\Psi)$ decomposes as the direct sum $C(\Psi_0)\oplus C(\Psi_1) \oplus \cdots \oplus C(\Psi_k)$, where $C(\Psi_0)$ is finite and $C(\Psi_1),\ldots, C(\Psi_k)$ infinite. Also, notice that there exists some $1 \leq i \leq k$ such that $\Psi_i$ contains at least three vertices. Otherwise, each $\Psi_i$ would be the disjoint union of two isolated vertices, and $C(\Psi)$ would be virtually abelian, which is impossible since it contains the free product $C(\Phi)$ of at least three copies of $\mathbb{Z}/2\mathbb{Z}$. Fix such an index $i$. Because $\Psi_i$ does not decompose as a join, its opposite graph $\Psi_i^{\mathrm{opp}}$ (i.e., the graph whose vertices are the vertices of $\Psi_i$ and whose edges link two vertices of $\Psi_i$ if they are not linked by an edge) must be connected. Fix two adjacent vertices $u,v$ of $\Psi_i^{\mathrm{opp}}$. Since $\Psi_i$ contains at least three vertices, there must exist a third vertex $w$ adjacent to either $u$ or $v$. Without loss of generality, suppose that we are in the latter case. Let $\Lambda$ denote the subgraph of $\Psi$ spanned by $u,v,w$. Notice that either $\Lambda$ is a disjoint union of three isolated vertices (if $u$ and $w$ are not adjacent) or the disjoint union of an isolated vertex with an isolated edge (if $u$ and $w$ are adjacent). It is clear that $C(\Lambda)$ contains a peripheral collection of $n$ pairwise non-transverse hyperplanes, and because $C(\Lambda)$ naturally embeds into $C(\Psi)$ as a convex subcomplex, it follows that $C(\Psi)$ also must contain such a collection.

\medskip \noindent
Thus, we have proved that, if $C(\Phi)$ is isomorphic to a subgroup of $C(\Psi)$, then $X(\Psi)$ contains a peripheral collection of hyperplanes whose crossing graph is isomorphic to $\Phi$. The converse follows from Theorem \ref{thm:ReflectionGeneral}.
\end{proof}

\noindent
However, this solution is not completely satisfying as we need to check infinitely many configurations of hyperplanes in order to show that an embedding is impossible. This problem will be addressed in Section \ref{section:peripheral}. Nevertheless, Corollary \ref{cor:Sub} has nice applications. Some of them are given in the next section.

\subsection{A few applications}\label{section:Appli}

\noindent
This subsection is dedicated to consequences of Corollary \ref{cor:Sub}. First, we prove the converse of Proposition \ref{prop:EmbeddingExTree}, providing a solution to the embedding problem among right-angled Coxeter groups defined by finite trees.

\begin{thm}\label{thm:EmbeddingTree}
Let $R,S$ be two finite trees. Then $C(R)$ is isomorphic to a subgroup of $C(S)$ if and only if there exists a graph morphism $\varphi : R \to S$ which sends a vertex of degree $2$ to a vertex of degree $\geq 2$ and a vertex of degree $\geq 3$ to a vertex of degree $\geq 3$. 
\end{thm}

\noindent
We emphasize that the embedding problem for right-angled Coxeter groups defined by finite forests reduces to the case of trees as a consequence of Corollary \ref{cor:ForestFromTree}. 

\begin{proof}[Proof of Theorem \ref{thm:EmbeddingTree}.]
Assume that $C(R)$ is isomorphic to a subgroup of $C(S)$. According to Corollary \ref{cor:Sub}, $X(S)$ contains a peripheral collection of hyperplanes $\mathcal{J}$ whose crossing graph is $R$. Let $\varphi : R \to S$ denote the map sending each vertex $r$ of $R$ to the vertex of $S$ which labels the hyperplanes $J_r$ of $\mathcal{J}$ corresponding to $r$. It follows from Lemma \ref{lem:transverseimpliesadj} that $\varphi$ is a graph morphism. If $r \in V(R)$ has degree at least two, then there exist two distinct and non-transverse hyperplanes of $\mathcal{J}$ which are transverse to $J_r$. This observation implies that $J_r$ is an unbounded hyperplane of $X(S)$, which implies that the vertex of $S$ labelling $J_r$ must have degree at least two as well. If $r$ has degree at least three, then there exist three pairwise distinct and non-transverse hyperplanes of $\mathcal{J}$ which are transverse to $J_r$. Notice that, because $\mathcal{J}$ is peripheral, none of these three hyperplanes separates one from another other, i.e., they define a \emph{facing triple}. As a consequence, the vertex of $S$ which labels $J_r$ must have degree at least three, since otherwise $N(J_r)$ would be the product of an edge $[0,1]$ with either another edge $[0,1]$ or a bi-infinite line so that $N(J_r)$ would not contain a facing triple. Thus, we have proved that $\varphi : R \to S$ is a graph morphism which sends a vertex of degree $2$ to a vertex of degree $\geq 2$ and a vertex of degree $\geq 3$ to a vertex of degree $\geq 3$. The converse is proved by Proposition \ref{prop:EmbeddingExTree}.
\end{proof}

\noindent
The next example provides an explicit application of Theorem \ref{thm:EmbeddingTree}. (See also Example~\ref{ex:GraphMorphism}.)

\begin{figure}
\begin{center}
\includegraphics[scale=0.45]{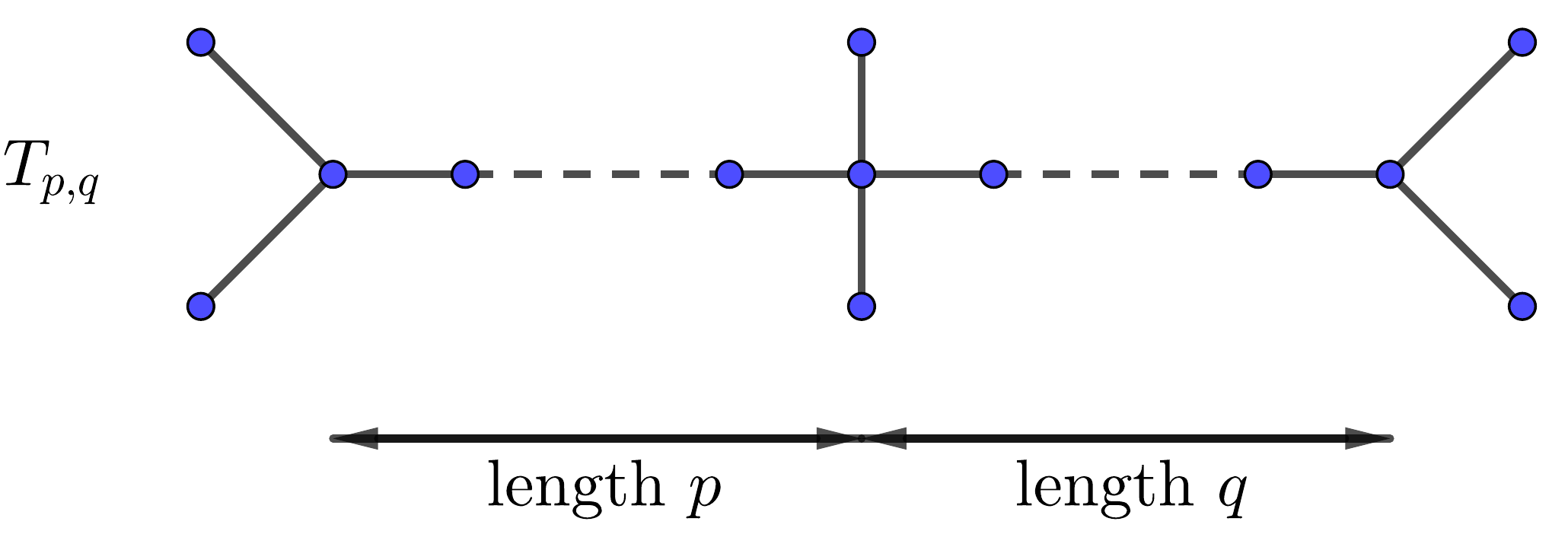}
\caption{The graph $T_{p,q}$ from Example \ref{ex:EmbeddingTree}.}
\label{TPQ}
\end{center}
\end{figure}
\begin{ex}\label{ex:EmbeddingTree}
For every $p,q \geq 1$, let $T_{p,q}$ denote the tree illustrated by Figure \ref{TPQ}. Fixing some $p,q,r,s \geq 1$, we want to determine when $C(T_{p,q})$ is isomorphic to a subgroup of $C(T_{r,s})$. We distinguish several cases.
\begin{enumerate}
	\item Assume that $p$ and $q$ are both even. The graph morphism $(1)$ given by Figure~\ref{Graphs} can be easily generalised to a graph morphism $T_{p,q} \to T_{r,s}$ which satisfies the assumption of Theorem \ref{thm:EmbeddingTree}. So $C(T_{p,q})$ embeds into $C(T_{r,s})$.
	\item Assume that $r$ and $s$ are both even but that $p$ or $q$ is odd. A graph morphism $T_{p,q} \to T_{r,s}$ satisfying the assumptions of Theorem \ref{thm:EmbeddingTree} has to send two vertices of degree $\geq 3$ at odd distance apart to two vertices of degree $\geq 3$ at odd distance apart. So such a graph morphism cannot exist, and we conclude that $C(T_{p,q})$ does not embed into $C(T_{r,s})$. 
	\item Assume that $p$ is odd, $q$ is even, $r$ is odd, and $s$ is even. From the general observation made in case 2, combined with the fact that a graph morphism is $1$-Lipschitz, it follows that $C(T_{p,q})$ does not embed into $C(T_{r,s})$ if $p<r$. If $p \geq r$, then the graph morphism $(2)$ given by Figure \ref{Graphs} can be easily generalised to a graph morphism $T_{p,q} \to T_{r,s}$ which satisfies the assumption of Theorem \ref{thm:EmbeddingTree}. So $C(T_{p,q})$ embeds into $C(T_{r,s})$.
	\item Assume that $p$ is odd, $q$ is even, $r$ and $s$ are both odd. The same argument as in the previous case shows that $C(T_{p,q})$ embeds into $C(T_{r,s})$ if and only if $p \geq \min(r,s)$.
	\item Assume that $p$ and $q$ are both odd, $r$ is odd, and $s$ is even. The same argument as in case 3 shows that $C(T_{p,q})$ does not embed into $C(T_{r,s})$ if $\max(p,q)<r$. Otherwise, if $\max(p,q) \geq r$, the graph morphism $(3)$ given by Figure \ref{Graphs} can be easily generalised to a graph morphism $T_{p,q} \to T_{r,s}$ which satisfies the assumption of Theorem \ref{thm:EmbeddingTree}. So $C(T_{p,q})$ embeds into $C(T_{r,s})$.
	\item Assume that $p,q,r,s$ are all odd. The same argument as in case 3 shows that $C(T_{p,q})$ does not embed into $C(T_{r,s})$ if $\max(p,q)< \min(r,s)$. Otherwise, if $\max(p,q) \geq \min(r,s)$, the graph morphism $(4)$ given by Figure \ref{Graphs} can be easily generalised to a graph morphism $T_{p,q} \to T_{r,s}$ which satisfies the assumption of Theorem \ref{thm:EmbeddingTree}. So $C(T_{p,q})$ embeds into $C(T_{r,s})$.
\end{enumerate}
Every possible configuration is symmetric (up to switching $p$ and $q$ or $r$ and $s$) to one the cases consided above.
\begin{figure}
\begin{center}
\includegraphics[scale=0.35]{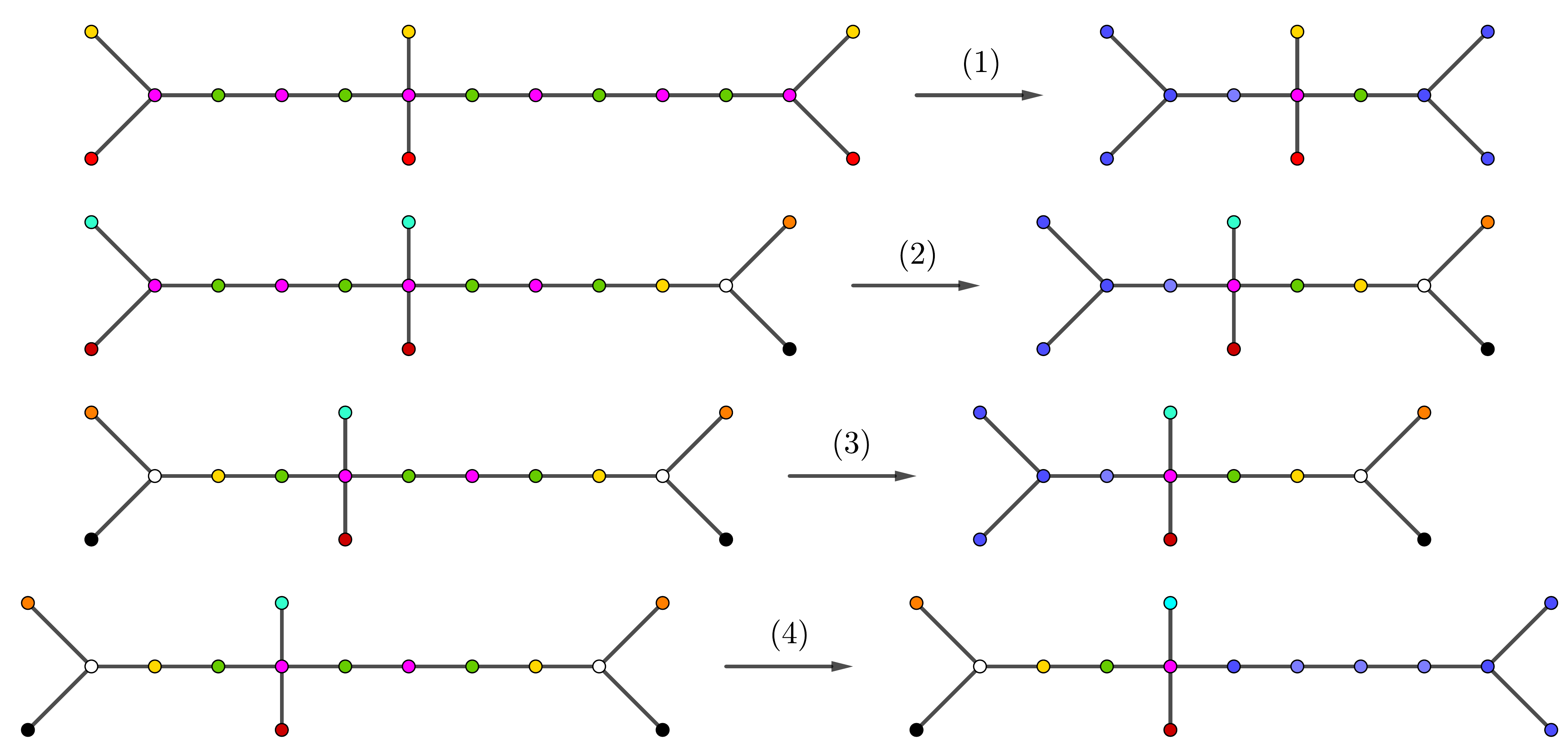}
\caption{Graph morphisms from Example \ref{ex:EmbeddingTree}.}
\label{Graphs}
\end{center}
\end{figure}
\end{ex}

\noindent
In our next statement, we determine which right-angled Coxeter groups embed into the right-angled Coxeter group defined by a cycle. In particular, one gets a converse of Corollary \ref{cor:EmbeddingExCycle}. 

\begin{prop}
Let $\Gamma$ be a finite simplicial graph and let $C_n$ denote the cycle of length $n \geq 5$. Then $C(\Gamma)$ is isomorphic to a subgroup of $C(C_n)$ if and only if $\Gamma$ is either a disjoint union of segments or a single cycle of length divisible by $n-4$. 
\end{prop}

\begin{proof}
Assume that $C(\Gamma)$ is isomorphic to a subgroup of $C(C_n)$. According to Corollary~\ref{cor:Sub}, there exists a peripheral collection of hyperplanes $\mathcal{H}$ in $X(C_n)$ whose crossing graph is $\Gamma$. It is well-known that $X(C_n)$ is a square tessellation of the hyperbolic plane; in particular, it is a planar CAT(0) square complex. As a consequence, if $\Gamma$ contains a cycle, then the corresponding cycle of hyperplanes of $\mathcal{H}$ has to enclose a region of $X(C_n)$ which contains the vertex $1$. As $\mathcal{H}$ is peripheral, this observation implies that $\mathcal{H}$ does not contain any other hyperplane. Thus, we have proved that $\Gamma$ is either a forest or a single cycle. The following claim shows that each vertex of such a forest must have degree at most two, or equivalently, that it must be a disjoint union of segments.

\begin{claim}
The group $(\mathbb{Z}/2\mathbb{Z}) \oplus ((\mathbb{Z}/2\mathbb{Z}) \ast (\mathbb{Z}/2\mathbb{Z}) \ast (\mathbb{Z}/2\mathbb{Z}))$ is not isomorphic to a subgroup of $C(C_n)$.
\end{claim}

\noindent
If it were, according to Corollary \ref{cor:Sub} there would exist a peripheral collection of four hyperplanes $\{A,B,C,D\}$ such that $A$ is transverse to $B$, $C$ and $D$, and such that $B,C,D$ are pairwise non-transverse. But the planarity of $X(C_n)$ implies that one hyperplane among $B,C,D$ has to separate the other two, contradicting the fact that the collection is peripheral.

\begin{claim}
If $C(C_m)$ is isomorphic to a subgroup of $C(C_n)$ for some $m \geq 5$, then $n-4$ divides $m-4$.
\end{claim}

\noindent
According to Corollary \ref{cor:Sub}, there exists a peripheral collection of hyperplanes $\mathcal{J}$ in $X(C_n)$ whose crossing graph is a cycle of length $m$. It is well-known that $X(C_n)$ is a square tessellation of the hyperbolic plane, so these hyperplanes are bi-infinite lines delimiting a region which contains $1$. If this region does not contain an edge, then the carriers of the hyperplanes of $\mathcal{J}$ pairwise intersect, and we deduce from Lemmas \ref{lem:transverseimpliesadj} and~\ref{lem:transverseiff} that the vertices of $C_n$ labelling the hyperplanes of $\mathcal{J}$ span an induced subgraph of $C_n$ isomorphic to $C_m$, hence $m=n$. Otherwise, if the region contains an edge, then the hyperplane dual to this edge defines a bi-infinite line which is transverse to exactly two hyperplanes of $\mathcal{J}$ and which cuts our region into two new regions corresponding to two other peripheral collections of hyperplanes whose crossing graphs are cycles. Let $p,q$ denote the lengths of these cycles. Notice that $m= p+q-4$. By induction, we may suppose that $n-4$ divides $p-4$ and $q-4$, so that $m-4= (p-4)+(q-4)$ also must be divisible by $n-4$. This concludes the proof of our claim. 

\medskip \noindent
Thus, we have proved that, if $C(\Gamma)$ is isomorphic to a subgroup of $C(C_n)$ then $\Gamma$ is either a disjoint union of segments or a single cycle of length divisible by $n-4$. The converse follows from Corollary \ref{cor:EmbeddingExCycle}.
\end{proof}

\noindent
We collect a few other observations in the following two propositions. In the first proposition, we refer to the \emph{double-star} as the smallest tree containing two adjacent vertices of degree $3$. 

\begin{prop}
Let $\Gamma$ be a triangle-free simplicial graph. The following assertions are equivalent:
\begin{itemize}
	\item[(i)] $\Gamma$ contains two adjacent vertices of degree $\geq 3$;
	\item[(ii)] $C(\Xi)$ is isomorphic to a subgroup of $C(\Gamma)$, where $\Xi$ denotes the double-star;
	\item[(iii)] for every finite forest $F$, $C(F)$ is isomorphic to a subgroup of $C(\Gamma)$.
\end{itemize}
\end{prop}

\begin{proof}
The implication $(i) \Rightarrow (iii)$ follows from Proposition \ref{cor:EmbeddingExTree} and $(iii) \Rightarrow (ii)$ is clear. 

\medskip \noindent
Now, suppose that $C(\Xi)$ is isomorphic to a subgroup of $C(\Gamma)$. According to Corollary~\ref{cor:Sub}, $X(\Gamma)$ contains a peripheral collection of hyperplanes $\mathcal{J}$ whose crossing graph is isomorphic to $\Xi$. Let $J$ and $H$ be the two hyperplanes corresponding to the two interior vertices $a$ and $b$ of $\Xi$, and let $x,y \in V(\Gamma)$ denote the two vertices of $\Gamma$ labelling $J$ and $H$ respectively. Notice that, because $J$ and $H$ are transverse, it follows from Lemma \ref{lem:transverseimpliesadj} that $x$ and $y$ are adjacent in $\Gamma$. We claim that $x$ and $y$ have degree $\geq 3$, which will conclude the proof of the implication $(ii) \Rightarrow (i)$. 

\medskip \noindent
Because $\mathcal{J}$ is peripheral, the three neighbors of $a$ in $\Xi$ must correspond to three hyperplanes transverse to $J$ such no one separates the other two, namely they define a \emph{facing triple}. This implies that $\mathrm{link}(x)$ cannot be empty (otherwise $N(J)$ would be a single edge) nor a single vertex (otherwise $N(J)$ would be a single square) nor two isolated vertices (otherwise $N(J)$ would decompose as the product of an edge with a bi-infinite line). Thus, $x$ must have degree at least three. The same argument shows that $y$ must have degree at least three as well, concluding the proof of our claim.
\end{proof}

\begin{prop}
Let $\Gamma$ be a triangle-free simplicial graph. For every $n \geq 0$, let $S_n$ denote the segment of length $n$. The following assertions are equivalent:
\begin{itemize}
	\item[(i)] The segment $S_3$ is an induced subgraph of $\Gamma$;
	\item[(ii)] $C(S_3)$ is isomorphic to a subgroup of $C(\Gamma)$;
	\item[(iii)] for every $n \geq 0$, $C(S_n)$ is isomorphic to a subgroup of $C(\Gamma)$.
\end{itemize}
\end{prop}

\begin{proof}
The implication $(i) \Rightarrow (iii)$ follows from Proposition \ref{cor:EmbeddingExSegment} and $(iii) \Rightarrow (ii)$ is clear.

\medskip \noindent
Now, suppose that $C(S_3)$ is isomorphic to a subgroup of $C(\Gamma)$. According to Corollary~\ref{cor:Sub}, $X(\Gamma)$ contains a peripheral collection of hyperplanes $\mathcal{J}$ whose crossing graph is isomorphic to $S_3$. Let $I_1, I_2$ be the two hyperplanes corresponding to the two interior vertices of $S_3$, and $J_1,J_2$ the two hyperplanes corresponding the endpoints of $S_3$ so that $J_1$ is transverse to $I_1$ and $J_2$ to $I_2$. Let $p$ denote the projection of $1$ onto $N(I_1) \cap N(I_2)$, and $p_1,p_2$ the projection of $p$ respectively onto $N(J_1),N(J_2)$. Assume that we chose our collection $\mathcal{J}$ so that $d(p,p_1)+d(p,p_2)$ is minimal.

\medskip \noindent
Notice that, because $X(\Gamma)$ is two-dimensional, the hyperplanes separating $p$ from $p_1$ (or $p_2$) are pairwise non-transverse. Let $H_1$ (resp. $H_2$) denote the last hyperplane separating $p$ from $p_1$ (resp. $p_2$). As a consequence of the minimality condition satisfied by $\mathcal{J}$, we know that $H_1$ is transverse to both $H_2$ and $J_2$, and that similarly $H_2$ is transverse to both $H_1$ and $J_1$. But $H_1$ is not transverse to $J_1$ as a consequence of Lemma \ref{lem:ProjSep}, and similarly $H_2$ is not transverse to $J_2$. Therefore, the crossing graph of $\{J_1,J_2,H_1,H_2\}$ is a segment of length three. Moreover, by construction, the carriers $N(J_1)$, $N(J_2)$, $N(H_1)$ and $N(H_2)$ pairwise intersect. We deduce from Lemmas \ref{lem:transverseimpliesadj} and~\ref{lem:transverseiff} that the four vertices of $\Gamma$ labelling $J_1,J_2,H_1,H_2$ span an induced segment of length three, concluding the proof of the implication $(ii) \Rightarrow (i)$. 
\end{proof}

\begin{prop}
Let $\Gamma$ be a triangle-free simplicial graph. For every $n \geq 3$, we denote by $C_n$ the cycle of length $n$. The following assertions are equivalent:
\begin{itemize}
	\item[(i)] The cycle $C_5$ is an induced subgraph of $\Gamma$;
	\item[(ii)] $C(C_5)$ is isomorphic to a subgroup of $C(\Gamma)$;
	\item[(iii)] for every $n \geq 3$, $C(C_n)$ is isomorphic to a subgroup of $C(\Gamma)$.
\end{itemize}
\end{prop}

\begin{proof}
The implication $(i) \Rightarrow (iii)$ follows from Corollary \ref{cor:EmbeddingExCycle} and $(iii) \Rightarrow (ii)$ is clear.

\medskip \noindent
Now, assume that $C(C_5)$ is isomorphic to a subgroup of $C(\Gamma)$. As a consequence of Corollary \ref{cor:Sub}, there exists a cycle $J_1, \ldots, J_5$ in the crossing graph of $X(\Gamma)$ such that $\{J_1, \ldots, J_5\}$ is peripheral. Without loss of generality, suppose that we chose our cycle in order to minimise the quantity $\displaystyle \sum\limits_{1 \leq i,j \leq 5} d(N(J_i),N(J_j))$. 

\medskip \noindent
Assume first that the vertex $1$ does not belong to the carrier of one of these five hyperplanes, say $J_1$. Then there exists a hyperplane $J$ separating $1$ from $N(J_1)$. Because $X(\Gamma)$ is two-dimensional, $J$ may be transverse to at most two hyperplanes among $J_1, \ldots, J_5$, and, because $J_1, \ldots, J_5$ is a cycle, $J$ cannot be transverse to exactly one hyperplane among $J_1, \ldots, J_5$. Moreover, if $J$ intersects exactly two hyperplanes among $J_1, \ldots, J_5$, say $J_2$ and $J_4$ (the other cases being $J_2$ and $J_5$, and $J_3$ and $J_5$), then $J$ separates $J_3$ from $J_1$ and $J_5$. Consequently, replacing $J_3$ with $J$ would create a new peripheral collection (possibly after a translation) such that
$$\sum\limits_{i \neq 3} d(N(J),N(J_i)) + \sum\limits_{i,j \neq 3} d(N(J_i),N(J_j)) < \sum\limits_{1 \leq i,j \leq 5} d(N(J_i),N(J_j)),$$
contradicting the choice of our cycle. Therefore, $J$ cannot be transverse to exactly two hyperplanes among $J_1, \ldots, J_5$ either. The only remaining case is that $J$ is transverse to none of the hyperplanes $J_1, \ldots, J_5$, i.e., $J$ separates $1$ from $J_1, \ldots, J_5$. Replacing our cycle with its image under a reflection along $J$ decreases the sum of the distances from $1$ to the carriers of $J_1, \ldots, J_5$, as a consequence of Lemma \ref{lem:shortendist}. Consequently, after finitely many iterations, we find a cycle of five hyperplanes such that all of them contain the vertex $1$ in their carriers. Such a cycle must be of the form $J_a, J_b, J_c, J_d, J_e$ where $a,b,c,d,e \in V(\Gamma)$. It follows from Lemmas \ref{lem:transverseimpliesadj} and \ref{lem:transverseiff} that $a,b,c,d,e$ defines a cycle of length five in $\Gamma$, concluding the proof of $(ii) \Rightarrow (i)$. 
\end{proof}

\section{The embedding problem in dimension two}

\subsection{Peripheral collections and cut-and-paste}\label{section:peripheral}

\noindent
We saw with Corollary \ref{cor:Sub} a characterisation of right-angled Coxeter groups $C(\Gamma_1)$ embeddable into a given two-dimensional right-angled Coxeter group $C(\Gamma_2)$. The goal of this section is to show that we only need to check finitely many configurations of hyperplanes in order to determine whether or not $C(\Gamma_1)$ is isomorphic to a subgroup of $C(\Gamma_2)$. We begin by introducing the following definition:

\begin{definition}
Let $\Gamma$ be a simplicial graph, and let $A$ and $B$ be two distinct and non-transverse hyperplanes of $X(\Gamma)$ in the same $C(\Gamma)$-orbit. Assume that $A$ separates $1$ from $B$. Let $p \in N(A)$ denote the projection of $1$ onto $\mathrm{proj}_{N(A)}(N(B))$, $q \in N(B)$ the projection of $p$ onto $N(B)$, and $r \in N(A)$ the neighbor of $p$ which is separated from it by $A$. The \emph{translation from $B$ to $A$} is the element $t(A,B) = rq^{-1}$.
\end{definition}

\noindent
The definition of $t(A,B)$ may seem to be obscure, but, as we shall see in the proof of Lemma \ref{lem:shortenagain} below, $t(A,B)$ is a product of the reflections along the hyperplanes separating $A$ and $B$ which sends the halfspace containing $1$ and delimited by $B$ to the halfspace containing $1$ and delimited by $A$. Loosely speaking, $t(A,B)$ is a canonical representative among the elements of $C(\Gamma)$ sending $B$ to $A$. 

\begin{lemma}\label{lem:shortenagain}
Let $A,B$ be two hyperplanes in the same $C(\Gamma)$-orbit such that $A$ separates $1$ from $B$, and set $t=t(A,B)$. Then:
\begin{itemize}
	\item $t$ sends the halfspace containing $1$ and delimited by $B$ to the halfspace containing $1$ and delimited by $A$.
	\item For every hyperplane $J$ separated from $1$ by $B$, we have $d(1,N(tJ))< d(1,N(J))$. 
	\item For every hyperplane $J$ transverse to both $A$ and $B$, $t$ stabilises each halfspace delimited by $J$.
\end{itemize}
\end{lemma}

\begin{proof}
Let $p \in N(A)$ denote the projection of $1$ onto $\mathrm{proj}_{N(A)}(N(B))$, $q \in N(B)$ the projection of $p$ onto $N(B)$, and $r \in N(A)$ the neighbor of $p$ which is separated from it by $A$. By definition of $t(A,B)$, we have $t=rq^{-1}$. Let $u \in V(\Gamma)$ denote the common label of $A$ and $B$. The first point of our lemma follows from the observation that $rq^{-1}$ sends the edge $[q,qu]$ (which is dual to $B$) to the edge $[r,ru]$ (which is dual to $A$).

\medskip \noindent
Next, fix a geodesic $\gamma$ from $q$ to $r$ passing through $p$, and let $s_1,\ldots, s_k$ denote the reflections of $C(\Gamma)$ associated to the edges of $\gamma$ (read from $q$ to $r$). Notice that 
$$s_k \cdots s_2s_1 \cdot q =r = t \cdot q \ \text{hence} \ \ t= s_k \cdots s_1.$$
Notice that, because $p$ and $q$ minimises the distance between $N(A)$ and $N(B)$, every hyperplane separating $p$ and $q$ separates $N(A)$ and $N(B)$ according to \cite[Proposition 2.7]{coningoff}. Therefore, $s_1, \ldots, s_{k-1}$ are reflections along hyperplanes separating $A$ and $B$. So, if $J$ is transverse to both $A$ and $B$, then it has to be transverse to all these hyperplanes as well, which implies that $s_1, \ldots, s_{k-1}$ all stabilise the halfspaces delimited by $J$. This proves the third point of our lemma.

\medskip \noindent
Finally, let $J$ be a hyperplane separated from $1$ by $B$. A key observation is that $q$ coincides with the projection of $1$ onto $N(B)$. Consequently,
$$\begin{array}{lcl} d(1,N(J)) & = & d(1,r)+d(r,q)+d(q,N(J)) = d(1,r)+d(r,q) + d(t \cdot q, N(tJ)) \\ \\ & \geq & d(1,r)+d(r,N(tJ)) +1 > d(1,N(tJ)), \end{array}$$
proving the second point of our lemma.
\end{proof}

\noindent
The key definition of the section is:

\begin{definition}\label{def:cut}
Let $\Gamma$ be a simplicial graph and $\mathscr{C}$ a collection of hyperplanes of $X(\Gamma)$. Fix two distinct non-transverse hyperplanes $A,B$ of $X(\Gamma)$ in the same $C(\Gamma)$-orbit. Assume that $A$ and $B$ are transverse to the same hyperplanes of $\mathscr{C}$, that no hyperplane of $\mathscr{C}$ lies between $A$ and $B$, and that $A$ separates $1$ from $B$. If $\mathscr{C}_0$ denotes the subcollection of $\mathscr{C}$ containing the hyperplanes separated from $1$ by $B$, one says that $(\mathscr{C} \backslash \mathscr{C}_0) \cup t(A,B) \mathscr{C}_0$ is obtained from $\mathscr{C}$ by a \emph{cut-and-paste}. 
\end{definition}

\noindent
This definition may seem technical, but the idea behind it is quite simple. Let $Y$ denote the intersection of all the halfspaces containing $1$ delimited by hyperplanes of $\mathscr{C}$ and let $Q$ denote its intersection with the subspace delimited by $A$ and $B$. The key observation is that $Q$ decomposes as a product. This follows from the fact that $Q$ lies inside the \emph{bridge} of $A$ and $B$ (i.e., the convex hull of the union of all the pairs of vertices minimising the distance between $N(A)$ and $N(B)$), which itself decomposes as a product $T \times I(x,y)$ where $T \subset N(A)$ is a convex subcomplex whose hyperplanes coincide with the hyperplanes crossing both $A$ and $B$ and where $I(x,y)$ is the union of all the geodesics between two vertices $x \in N(A)$ and $y \in N(B)$ minimising the distance between $N(A)$ and $N(B)$. We refer to \cite[Lemma 2.18]{CIF} for more information on bridges. Consequently, $Q$ decomposes as the product $(Q \cap N(A)) \times I(x,y)$. Moreover, if $z$ denotes the vertex adjacent to $y$ and separated from it by $B$, then $t(A,B)=xz^{-1}$. Notice that $t(A,B)$ acts on $Q = (Q \cap N(A)) \times I(x,y)$ as a translation along the second coordinate. Therefore, when replacing $\mathscr{C}_0$ with $t(A,B) \mathscr{C}_0$, what we are doing is cutting $Y$ along $A$ and $B$, removing the middle part $Q$, and gluing $N(B)$ with $N(A)$. See Figure~\ref{cut}. 
\begin{figure}
\begin{center}
\includegraphics[scale=0.38]{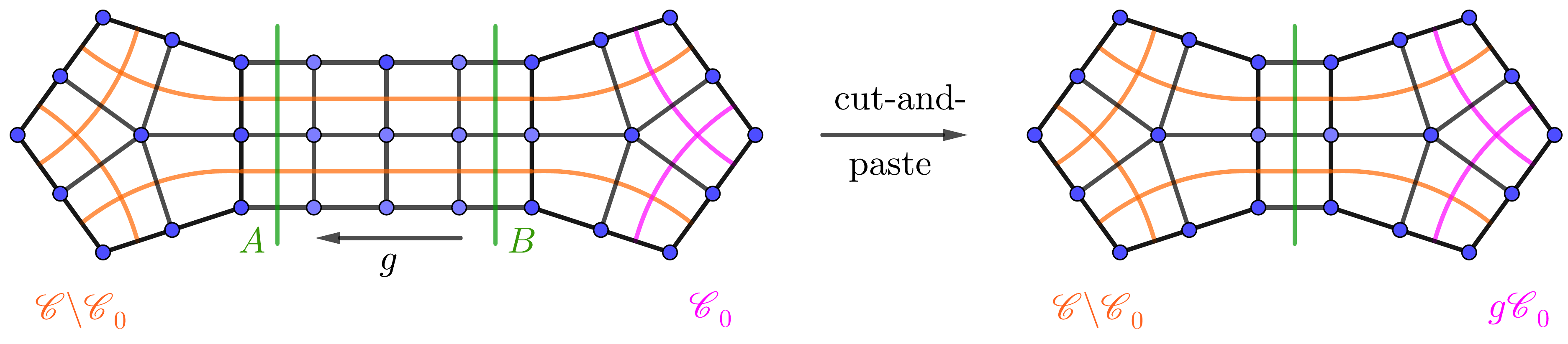}
\caption{Example of a cut-and-paste.}
\label{cut}
\end{center}
\end{figure}

\medskip \noindent
First of all, we claim that a cut-and-paste creates a peripheral collection from a peripheral collection with the same crossing graph. 

\begin{lemma}\label{lem:CutCrossing}
Let $\Gamma$ be a simplicial graph and $\mathscr{C}, \mathscr{C}'$ two collections of hyperplanes of $X(\Gamma)$. Assume that $\mathscr{C}$ is peripheral and that $\mathscr{C}'$ can be obtained from $\mathcal{C}$ by a cut-and-paste. Then $\mathscr{C}'$ is also peripheral and its crossing graph is isomorphic to the crossing graph of $\mathscr{C}$.
\end{lemma}

\begin{proof}
Because $\mathscr{C}'$ can be obtained from $\mathscr{C}$ by cut-and-paste, there exist two non-transverse hyperplanes $A,B$ in the same $C(\Gamma)$-orbit which are transverse to the same hyperplanes of $\mathscr{C}$, such that no hyperplane of $\mathscr{C}$ lie between $A$ and $B$, and such that $A$ separates $1$ from $B$, so that, if $\mathscr{C}_0$ denotes the subcollection of $\mathscr{C}$ containing the hyperplanes separated from $1$ by $B$, then $\mathscr{C'}=(\mathscr{C} \backslash \mathscr{C}_0) \cup t(A,B) \mathscr{C}_0$. We claim that, if $\mathscr{C}'$ is not peripheral, then $\mathcal{C}$ cannot be peripheral as well.

\medskip \noindent
So let $J,H \in \mathscr{C}'$ be two hyperplanes such that $J$ separates $1$ from $H$. 

\medskip \noindent
First, assume that $J \in t(A,B) \mathscr{C}_0$. Because $A$ separates $1$ from $J$, necessarily $A$ also separates $1$ from $H$. Consequently, if $H \in \mathscr{C}$ then $B$ has to separate $1$ from $H$, because $H$ cannot lie between $A$ and $B$ and because if $H$ is transverse to $B$ then it has to be transverse to $A$ as well. Hence $H \in \mathscr{C}_0$. Consequently, because $H$ cannot belong to $\mathscr{C} \backslash \mathscr{C}_0$, it has to belong to $t(A,B) \mathscr{C}_0$. So $t(A,B)^{-1}J$ separates $1$ from $t(A,B)^{-1}H$, and since these two hyperplanes belong to $\mathscr{C}_0 \subset \mathscr{C}$, it follows that $\mathscr{C}$ is not peripheral.

\medskip \noindent
Next, assume that $J \in \mathscr{C} \backslash \mathscr{C}_0$. It is clear that, if $H$ also belongs to $\mathscr{C} \backslash \mathscr{C}_0$, then $\mathscr{C}$ is not peripheral. So assume that $H \in t(A,B) \mathscr{C}_0$. If $J$ is transverse to $A$, it follows from Lemma~\ref{lem:shortenagain} that $t(A,B)^{-1}$ stabilises $J$ and leave $H$ inside the halfspace delimited by $J$ which does not contain $1$. Because $t(A,B)^{-1}H$ belongs to $\mathscr{C}_0 \subset \mathscr{C}$, it follows that $\mathscr{C}$ is not peripheral. Now, assume that $J$ is not transverse to $A$. Necessarily, $A$ separates $J$ from $H$, so $B$ separates $J$ from $t(A,B)^{-1}H$. We also conclude that $\mathscr{C}$ is not peripheral. 

\medskip \noindent
Thus, we have proved that, if $\mathscr{C}$ is peripheral, then $\mathscr{C}'$ is peripheral as well. It remains to show that $\mathscr{C}$ and $\mathscr{C}'$ have isomorphic crossing graphs. Set
$$\varphi : \left\{ \begin{array}{ccc} \mathscr{C} & \to & \mathscr{C}' \\ J & \mapsto & \left\{ \begin{array}{cl} J & \text{if $J \notin \mathscr{C}_0$} \\ t(A,B) J & \text{if $J \in \mathscr{C}_0$} \end{array} \right. \end{array} \right..$$
We claim that $\varphi$ induces an isomorphism between the crossing graphs of $\mathscr{C}$ and $\mathscr{C}'$. Let $J,H \in \mathscr{C}$ be two hyperplanes. If they both belong to $\mathscr{C} \backslash \mathscr{C}_0$ or both to $\mathscr{C}_0$, it is clear that $J$ and $H$ are transverse if and only if $\varphi(J)$ and $\varphi(H)$ are transverse. So without loss of generality, assume that $J$ belongs to $\mathscr{C} \backslash \mathscr{C}_0$ and $H$ to $\mathscr{C}_0$. If $J$ is not transverse to $B$, then $J$ and $H$ cannot be transverse, and $A$ separates $J$ and $t(A,B) H$ so that $\varphi(J)=J$ and $\varphi(H)= t(A,B) H$ cannot be transverse either. If $J$ is transverse to $B$, it also has to be transverse to $A$, and Lemma \ref{lem:shortenagain} implies that $H$ is transverse to $J$ if and only if $t(A,B) H=\varphi(H)$ is transverse to $J=\varphi(J)$. This concludes the proof.
\end{proof}

\noindent
The main statement of this section is the following proposition:

\begin{prop}\label{prop:MainCut}
Let $\Gamma$ be a finite simplicial graph and $\mathscr{C}$ a finite peripheral collection of hyperplanes of $X(\Gamma)$. There exists a collection $\mathscr{C}'$ obtained from $\mathscr{C}$ by cuts-and-pastes such that each hyperplane of $\mathscr{C}'$ intersects the ball of radius $2(1+(1+2 \cdot \# \mathscr{C}) \cdot \# V(\Gamma) )$ centered at the vertex $1$. 
\end{prop}

\begin{proof}
Assume that there exists a hyperplane $J \in \mathscr{C}$ which does not cross the ball centered at $1$ of radius $2(1+(1+2 \cdot \# \mathscr{C}) \cdot \# V(\Gamma) )$. Let $J_1, \ldots, J_k$ be a maximal collection of pairwise non-transverse hyperplanes separating $1$ from $J$. Without loss of generality, suppose that $J_i$ separates $1$ from $J_{i+1}$ for every $1 \leq i \leq k-1$. Notice that, for every $1 \leq i \leq k-1$, $J_i$ and $J_{i+1}$ are transverse, and that, as a consequence of Lemma \ref{lem:dinfty}, $k > (1+2 \cdot \# \mathscr{C}) \cdot \# V(\Gamma)$. For every $1 \leq i \leq k$, let $\mathscr{C}_i$ denote the set of the hyperplanes of $\mathscr{C}$ crossing the subspace delimited by $J_i$ and $J_{i+1}$. 

\medskip \noindent
If $\mathscr{C}_i \neq \mathscr{C}_{i+1}$ for some $1 \leq i \leq k-1$, then either one hyperplane of $\mathscr{C}_i$ is not transverse to $J_{i+1}$, and if so such a hyperplane cannot belong to $\mathscr{C}_j$ for $j \geq i+1$; or one hyperplane of $\mathscr{C}_{i+1}$ is not transverse to $J_{i+1}$, and if so such a hyperplane cannot belong to $\mathscr{C}_j$ for $j \leq i$. Consequently, there exist at most $2 \cdot \# \mathscr{C}$ indices $i$ such that $\mathscr{C}_i \neq \mathscr{C}_{i+1}$, which implies that there exist at least $1+\#V(\Gamma)$ consecutive indices for which the set $\mathscr{C}_i$ remains the same.

\medskip \noindent
In other words, there exist two indices $1 \leq i < j \leq k$ satisfying $j \geq i+ \# V(\Gamma)$ such that the hyperplanes $J_i, \ldots, J_j$ intersect exactly the same hyperplanes of $\mathscr{C}$ and such that no hyperplane of $\mathscr{C}$ lies between two hyperplanes among $J_i, \ldots, J_j$. Because the number of $C(\Gamma)$-orbits of hyperplane of $X(\Gamma)$ is $\#V(\Gamma)$, there must exist two indices $i \leq r< s \leq j$ such that $J_r$ and $J_s$ belong to the same $C(\Gamma)$-orbit. So, if we denote by $\mathscr{C}_0$ the subcollection of $\mathscr{C}$ containing the hyperplanes separated from $1$ by $J_s$, then $\mathscr{C}'= (\mathscr{C} \backslash \mathscr{C}_0) \cup  t(J_r,J_s) \mathscr{C}_0$ is obtained from $\mathscr{C}$ by a cut-and-paste. Moreover, it follows from Lemma \ref{lem:shortenagain} that
$$\sum\limits_{J \in \mathscr{C}'} d(1,N(J)) < \sum\limits_{J \in \mathscr{C}} d(1,N(J)).$$
Therefore, by iterating the process, we eventually find a collection of hyperplanes as desired. 
\end{proof}

\noindent
The combination of Corollary \ref{cor:Sub} with Proposition \ref{prop:MainCut} yields the following statement:

\begin{cor}\label{cor:SubAlgo}
Let $\Phi,\Psi$ be two finite simplicial graphs. Assume that $\Psi$ is triangle-free. Then $C(\Phi)$ is isomorphic to a subgroup of $C(\Psi)$ if and only if $X(\Psi)$ contains a peripheral collection of hyperplanes whose crossing graph is isomorphic to $\Phi$ and all of whose hyperplanes intersect the ball of radius $2(1+(1+2 \cdot \# V(\Phi)) \cdot \# V(\Psi) )$ centered at the vertex $1$. 
\end{cor}\qed

\subsection{The algorithm}\label{section:AlgoFinal}

\noindent
In this section, we show how to use Corollary \ref{cor:SubAlgo} in order to determine algorithmically whether or not one two-dimensional right-angled Coxeter group is isomorphic to a subgroup of another one. 

\begin{thm}\label{thm:Algo}
There exists an algorithm determining, given two finite and triangle-free simplicial graphs $\Phi, \Psi$, whether or not $C(\Phi)$ is isomorphic to a subgroup of $C(\Psi)$. If so, an explicit basis is provided. 
\end{thm}

\noindent
A \emph{basis} of a right-angled Coxeter group is defined as follows:

\begin{definition}
Let $\Gamma$ be a simplicial graph. A \emph{basis} of $C(\Gamma)$ is a generating set $S \subset C(\Gamma)$ such that, if $\Delta$ denotes the graph whose vertex-set is $S$ and whose edges link two commuting elements, then the map $C(\Delta) \to C(\Gamma)$ defined by sending a vertex of $\Delta$ to the element of $S$ it represents induces an isomorphism. 
\end{definition}

\noindent
We begin by showing that a few elementary decision problems can be solved algorithmically:

\begin{lemma}\label{lem:AlgoId}
Let $\Gamma$ be a simplicial graph, $g,h \in C(\Gamma)$ two words of generators and $u,v \in V(\Gamma)$ two vertices. The hyperplanes $gJ_u$ and $hJ_v$ coincide if and only if $u=v$ and the reduction of $g^{-1}h$ contains only vertices of $\mathrm{star}(u)$.
\end{lemma}

\begin{proof}
The lemma is an immediate consequence of the description of the hyperplanes of $X(\Gamma)$ provided by Lemma \ref{lem:DescriptionHyp}.
\end{proof}

\begin{lemma}\label{lem:AlgoTransverse}
Let $\Gamma$ be a simplicial graph, $g,h \in C(\Gamma)$ two words of generators and $u,v \in V(\Gamma)$ two vertices. The hyperplanes $gJ_u$ and $hJ_v$ are transverse if and only if $u$ and $v$ are adjacent vertices and the reduction of $g^{-1}h$ is the concatenation of a word whose letters are vertices of $\mathrm{star}(u)$ with a word whose letters are vertices of $\mathrm{star}(v)$.
\end{lemma}

\begin{proof}
Assume that  $gJ_u$ and $hJ_v$ are transverse. Then there exists a geodesic from $1 \in N(J_u)$ to $g^{-1}h \in N(g^{-1}hJ_v)$ which is the concatenation of a geodesic in $N(J_u)= \langle \mathrm{star}(u) \rangle$ with a geodesic in $N(g^{-1}hJ_v)$. As the edges of $N(g^{-1}hJ_v)$ are all labelled by vertices of $\mathrm{star}(v)$, it follows from Lemma \ref{lem:geod} that the reduction of $g^{-1}h$ is the concatenation of a word whose letters are vertices of $\mathrm{star}(u)$ with a word whose letters are vertices of $\mathrm{star}(v)$. We also know from Lemma \ref{lem:transverseimpliesadj} that $u$ and $v$ must be adjacent.

\medskip \noindent
Conversely, assume that $u$ and $v$ are adjacent and that $g^{-1}h$ belongs to $\langle \mathrm{star}(u) \rangle \langle \mathrm{star}(v) \rangle$. So there exist $a \in \langle \mathrm{star}(u) \rangle$ and $b \in \langle \mathrm{star}(v) \rangle$ such that $g^{-1}h=ab$. Because $u$ and $v$ are adjacent, the hyperplanes $J_u$ and $J_v$ must be transverse. It follows that $aJ_u=J_u$ and $aJ_v=abJ_v=g^{-1}hJ_v$ are transverse, and finally that $gJ_u$ and $hJ_v$ are transverse, concluding the proof.
\end{proof}

\begin{lemma}\label{lem:AlgoPeripheral}
Let $\Gamma$ be a simplicial graph, $g,h \in C(\Gamma)$ two words of generators and $u,v \in V(\Gamma)$ two vertices. Let $g'$ (resp. $h'$) denote the word of generators obtained from the reduction of $g$ (resp. $h$) by removing the letters of the tail of $g$ (resp. $h$) which are vertices of $\mathrm{star}(u)$ (resp. $\mathrm{star}(v)$). The hyperplane $gJ_u$ separates $1$ from $hJ_v$ if and only if $g'u$ is a prefix of $h'$ in $C(\Gamma)$.
\end{lemma}

\begin{proof}
Notice that $hJ_v = h'J_v$. Moreover, if a vertex $x$ belongs to $h'J_v$, then $x=h'y$ for some $y \in \langle \mathrm{star}(v) \rangle$, and the product $h'y$ must be reduced because the tail of $h'$ does not contain vertices of $\mathrm{star}(v)$. It follows from Lemma \ref{lem:geod} that $h'$ belongs to a geodesic between $1$ and $x$. In other words, $h'$ minimises the distance from $1$ in $N(hJ_v)$, i.e., $h'$ is the projection of $1$ onto $N(hJ_v)$. Therefore, as a consequence of Lemma \ref{lem:ProjSep}, $gJ_u$ separates $1$ from $hJ_v$ if and only if it separates $1$ from $h'$, which happens if and only if $g'u$ is a prefix of $h'$ according to Lemma \ref{lem:halfspaces}.
\end{proof}

\noindent
We are now ready to prove Theorem \ref{thm:Algo}.

\begin{proof}[Proof of Theorem \ref{thm:Algo}.]
For convenience, let $n$ denote the number of vertices of $\Phi$. Apply the following operations to a tuple $(g_1,\ldots, g_n, u_1, \ldots, u_n)$ where $g_1, \ldots, g_n \in C(\Psi)$ are words of generators of length $\leq 2(1+(1+2 \cdot \# V(\Phi)) \cdot \# V(\Psi) )$ and $u_1, \ldots, u_n \in V(\Psi)$:
\begin{itemize}
	\item For each $1 \leq i \leq n$, determine whether the word $g_i$ is reduced. If at some point the answer is no, stop and test another tuple; otherwise, apply the next step.
	\item For each $1 \leq i < j \leq n$, apply Lemma \ref{lem:AlgoId} to determine whether the hyperplanes $g_iJ_{u_i}$ and $g_jJ_{u_j}$ are distinct. If at some point the answer is no, stop and test another tuple; otherwise, apply the next step.
	\item For each $1 \leq i < j \leq n$, apply Lemma \ref{lem:AlgoPeripheral} to determine whether $g_iJ_{u_i}$ separates $1$ from $g_jJ_{u_j}$. If at some point the answer is yes, stop and test another tuple; otherwise, the collection $\mathcal{J}= \{ g_1J_{u_1}, \ldots, g_nJ_{u_n} \}$ is peripheral.
	\item Apply Lemma \ref{lem:AlgoTransverse} to determine the crossing graph of $\mathcal{J}$. If it is isomorphic to $\Phi$, the algorithm stops and gives $(g_1u_1g_1^{-1}, \ldots, g_nu_ng_n^{-1})$ as a basis of a subgroup of $C(\Psi)$ isomorphic to $C(\Phi)$ (as justified by Theorem \ref{thm:ReflectionGeneral}). Otherwise, test another tuple.
\end{itemize}
If all the tuples have been tested and the algorithm did not stop to give a basis of a subgroup, then it follows from Corollary \ref{cor:SubAlgo} that $C(\Psi)$ is not isomorphic to a subgroup of $C(\Phi)$.
\end{proof}

\begin{remark}
In the proof of Theorem \ref{thm:Algo}, we did not try to write the most optimal algorithm. A rough upper bound on the number of operations required for the most naive application of the process described there is 
$$5 \cdot 10^8 \cdot 8^p \cdot p^{14+p} \cdot q^{8+2p} \cdot \mathrm{iso}(p),$$ 
or $(pq)^{3p} \cdot \mathrm{iso}(p)$ for large values of $p,q$, where $\mathrm{iso}(p)$ corresponds to the number of operations needed in order to determine whether or not two graphs with $p$ vertices are isomorphic. 
\end{remark}

\noindent
Related to Theorem \ref{thm:Algo} is \cite[Theorem D]{RACGStallings}, which proves that there exists an algorithm determining whether or not a one-ended right-angled Coxeter group is isomorphic to a finite-index subgroup of another one. It turns out that this theorem also follows from the results proved in Sections \ref{section:MainThm} and \ref{section:peripheral}.

\begin{thm}\label{thm:AlgoFI}\emph{\cite{RACGStallings}}
There exists an algorithm determining, given two finite and triangle-free simplicial graphs $\Phi, \Psi$ where $\Phi$ has no isolated vertex, whether or not $C(\Phi)$ is isomorphic to a finite-index subgroup of $C(\Psi)$. If so, an explicit basis is provided and the index is computed. 
\end{thm}

\noindent
The link between the index of a subgroup and our peripheral collections of hyperplanes is made by the notion of \emph{covolume}:

\begin{definition}
Let $\Gamma$ be a simplicial graph and $\mathscr{C}$ a peripheral collection of hyperplanes of $X(\Gamma)$. The \emph{covolume} is the number of vertices (possibly infinite) in the intersection of all the halfspaces containing $1$ delimited by hyperplanes of $\mathscr{C}$. 
\end{definition}

\noindent
Notice that, as a consequence of Theorem \ref{thm:ReflectionGeneral}, this intersection of halfspaces $Y$ is a fundamental domain for the action on $X(\Gamma)$ of the subgroup $H$ generated by the reflections along hyperplanes of $\mathscr{C}$. Therefore, the covolume of $\mathscr{C}$ coincides with the index of $H$ in the right-angled Coxeter group $C(\Gamma)$.

\medskip \noindent
First of all, we need to study how the covolume may be modified by a cut-and-paste.

\begin{lemma}\label{lem:Covolume}
Let $\Gamma$ be a simplicial graph and $\mathscr{C}, \mathscr{C}'$ two collections of hyperplanes of $X(\Gamma)$. Assume that $\mathscr{C}$ is peripheral and that $\mathscr{C}'$ can be obtained from $\mathcal{C}$ by a cut-and-paste. If the covolume of $\mathscr{C}$ is infinite, then so is the covolume of $\mathscr{C}'$; otherwise, the covolume of $\mathscr{C}'$ is finite and smaller than the covolume of $\mathscr{C}$.
\end{lemma}

\begin{proof}
Because $\mathscr{C}'$ can be obtained from $\mathscr{C}$ by cut-and-paste, there exist two distinct non-transverse hyperplanes $A,B$ in the same $C(\Gamma)$-orbit which are transverse to the same hyperplanes of $\mathscr{C}$, such that no hyperplane of $\mathscr{C}$ lie between $A$ and $B$, and such that $A$ separates $1$ from $B$, and there exists an element $g \in C(\Gamma)$ which sends the halfspace delimited by $B$ containing $1$ to the halfspace delimited by $A$ containing $1$, such that, if $\mathscr{C}_0$ denotes the subcollection of $\mathscr{C}$ containing the hyperplanes separated from $1$ by $B$, then $\mathscr{C'}=(\mathscr{C} \backslash \mathscr{C}_0) \cup g \mathscr{C}_0$. Let $Y$ (resp. $Y'$) denote the intersection of all the halfspaces delimited by hyperplanes of $\mathscr{C}$ (resp. $\mathscr{C}'$) which contain $1$.

\medskip \noindent
Let $P$ denote the intersection of $Y$ with the halfspace delimited by $A$ which contains $1$, $Q$ the intersection of $Y$ with the subspace delimited by $A$ and $B$, and $R$ the intersection of $Y$ with the halfspace delimited by $B$ which does not contain $1$. Notice that $Y=P \sqcup Q \sqcup R$. Moreover, we have $Y'=P \cup gR$. As a consequence, if $Y$ is finite, then $Y'$ has to be finite as well and its cardinality is smaller than the cardinality of $Y$ (since the fact that $A$ and $B$ are distinct implies that $Q$ is non-empty). This proves the second assertion of our lemma. 

\medskip \noindent
In order to prove the first assertion, it suffices to show that, if $Q$ is infinite, then so is $P$. Notice that, because $Y$ is convex and because $A$ crosses $Y$, if $x$ is a vertex of $Q \cap N(A)$ then the vertex adjacent to $x$ separated from it by $J$ has to belong to $Q$. So, in fact, it is sufficient to show that $Q \cap N(A)$ is infinite. However, as observed right after Definition~\ref{def:cut}, $Q$ decomposes as a Cartesian product between $Q \cap N(A)$ and a finite subcomplex, which concludes the proof. 
\end{proof}

\noindent
As a consequence, we deduce the following analogue of Corollary \ref{cor:SubAlgo} for finite-index subgroups:

\begin{cor}\label{cor:SubAlgoFI}
Let $\Phi,\Psi$ be two finite simplicial graphs. Assume that $\Psi$ is triangle-free and that $\Phi$ has no isolated vertex. Then $C(\Phi)$ is isomorphic to a finite-index subgroup of $C(\Psi)$ if and only if $X(\Psi)$ contains a peripheral collection of hyperplanes whose covolume is finite, whose crossing graph is isomorphic to $\Phi$, and all of whose hyperplanes intersect the ball of radius $2(1+(1+2 \cdot \# V(\Phi)) \cdot \# V(\Psi) )$ centered at the vertex $1$. 
\end{cor}

\begin{proof}
Assume that $C(\Phi)$ is isomorphic to a finite-index subgroup of $C(\Psi)$. As a consequence of Corollary \ref{cor:Embedding}, there exists a peripheral collection of hyperplanes $\mathcal{J}$ such that the image of $C(\Phi)$ in $C(\Psi)$ is generated by the reflections along hyperplanes of $\mathcal{J}$. Necessarily, the crossing graph of $\mathcal{J}$ is isomorphic to $\Phi$ and its covolume is finite. By applying Proposition \ref{prop:MainCut}, we know that we can construct a new peripheral collection $\mathcal{H}$ by cuts-and-pastes all of whose hyperplanes intersect the ball of radius $2(1+(1+2 \cdot \# V(\Phi)) \cdot \# V(\Psi) )$ centered at the vertex $1$. According to Lemma \ref{lem:CutCrossing} and Corollary \ref{lem:Covolume}, the crossing graph of $\mathcal{H}$ is still isomorphic to $\Phi$ and its covolume is also finite (and smaller than the covolume of $\mathcal{J}$). 
\end{proof}

\noindent
Finally, when a peripheral collection of hyperplanes has finite covolume, we want to be able to compute its covolume algorithmically. This is the purpose of our next lemma.

\begin{lemma}\label{lem:ControlVolume}
Let $\Gamma$ be a finite simplicial graph and $\mathscr{C} = \{g_1J_{u_1}, \ldots, g_nJ_{u_n} \}$ a peripheral collection of hyperplanes of $X(\Gamma)$. Let $Y$ denote the intersection of all the halfspaces containing $1$ delimited by hyperplanes of $\mathscr{C}$. If $\mathscr{C}$ has finite covolume, then $Y$ lies in the ball centered at $1$ of radius $$\mathrm{clique}(\Gamma) + \sum\limits_{i=1}^n |g_i|,$$ where $\mathrm{clique}(\Gamma)$ denotes the maximal size of a complete subgraph of $\Gamma$. 
\end{lemma}

\begin{proof}
Fix a vertex $x \in Y$ such that $d(1,x)$ is maximal. Let $\mathcal{J}$ denote the collection of the hyperplanes of $\mathscr{C}$ which are tangent to $x$, and $p$ the projection of $1$ onto $\bigcap\limits_{J \in \mathcal{J}} N(J)$.

\medskip \noindent
First, we claim that $d(p,x) \leq \mathrm{clique}(\Gamma)$. Let $\Gamma_1$ denote the set of vertices of $\Gamma$ labelling the hyperplanes of $\mathcal{J}$. So, for every $u \in \Gamma_1$, a hyperplane of $\mathscr{C}$ separates $x$ from $xu$. But if $u \notin \Gamma_1$, then $xu$ still belongs to $Y$, hence $d(1,xu)<d(1,x)$ by our choice of $x$. Therefore, the product $xu$ cannot be reduced, or equivalently, $u$ belongs to the tail of $x$. Because the vertices of the tail of $x$ must be pairwise adjacent, it follows that the vertices of $\Gamma_1^c$ are pairwise adjacent. Now, the hyperplanes separating $x$ and $p$ must be transverse to each hyperplane of $\mathcal{J}$, so that Lemma \ref{lem:transverseimpliesadj} implies that they are labelled by vertices of $\Gamma_1^c$. Therefore, we can write $p=xa_1 \cdots a_k$ for some pairwise adjacent vertices $a_1, \ldots, a_k \in \Gamma$, which proves that $d(p,x) \leq k \leq \mathrm{clique}(\Gamma)$ as desired.

\medskip \noindent
Next, it follows from Lemma \ref{lem:DistInter} that 
$$d(1,p) \leq \sum\limits_{J \in \mathcal{J}} d(1,N(J)) \leq \sum\limits_{i=1}^n d(1,g_i) = \sum\limits_{i=1}^n |g_i|.$$
We conclude that
$$d(1,x) \leq d(1,p)+d(p,x) \leq \mathrm{clique}(\Gamma) + \sum\limits_{i=1}^n |g_i|,$$
as desired.
\end{proof}

\noindent
We now have everything we need in order to prove Theorem \ref{thm:AlgoFI}.

\begin{proof}[Proof of Theorem \ref{thm:AlgoFI}.]
For convenience, let $n$ denote the number of vertices of $\Phi$. Apply the following operations to a tuple $(g_1,\ldots, g_n, u_1, \ldots, u_n)$ where $g_1, \ldots, g_n \in C(\Psi)$ are words of generators of length $\leq 2(1+(1+2 \cdot \# V(\Phi)) \cdot \# V(\Psi) )$ and $u_1, \ldots, u_n \in V(\Psi)$:
\begin{itemize}
	\item For each $1 \leq i \leq n$, determine whether the word $g_i$ is reduced. If at some point the answer is no, stop and test another tuple; otherwise, apply the next step.
	\item For each $1 \leq i < j \leq n$, apply Lemma \ref{lem:AlgoId} to determine whether the hyperplanes $g_iJ_{u_i}$ and $g_jJ_{u_j}$ are distinct. If at some point the answer is no, stop and test another tuple; otherwise, apply the next step.
	\item For each $1 \leq i < j \leq n$, apply Lemma \ref{lem:AlgoPeripheral} to determine whether $g_iJ_{u_i}$ separates $1$ from $g_jJ_{u_j}$. If at some point the answer is yes, stop and test another tuple; otherwise, the collection $\mathcal{J}= \{ g_1J_{u_1}, \ldots, g_nJ_{u_n} \}$ is peripheral.
	\item Apply Lemma \ref{lem:AlgoTransverse} to determine the crossing graph of $\mathcal{J}$. If it is isomorphic to $\Phi$, apply the next step; otherwise, test another tuple.
	\item For each $1 \leq i \leq n$ and for each word of generators $x \in C(\Psi)$ of length $3+|g_1|+\cdots + |g_n|$, apply Lemma \ref{lem:halfspaces} to determine whether or not $g_iJ_{u_i}$ separates $1$ from $x$. If at some point the answer is no for a fixed $x$ and for each $i$, then the covolume of $\mathcal{J}$ is infinite (according to Lemma \ref{lem:ControlVolume}), so stop and test another tuple. Otherwise, apply the next step.
	\item Use Lemma \ref{lem:halfspaces} to list the words of generators $x \in C(\Psi)$ of length $\leq 2+|g_1|+ \cdots + |g_n|$ such that $g_iJ_{u_i}$ does not separate $1$ from $x$ for every $1 \leq i \leq n$. In that list $\{x_1, \ldots, x_k\}$, for every $1 \leq i \leq k$, if there exists $i < j \leq k$ such that $x_i$ and $x_j$ are equal in $C(\Psi)$, remove $x_j$. The cardinality $N$ of the final list we get is the covolume of $\mathcal{J}$. Stop the algorithm, and give $(g_1u_1g_1^{-1},\ldots, g_nu_ng_n^{-1})$ as the basis of a subgroup of $C(\Psi)$ of index $N$ which is isomorphic to $C(\Phi)$. 
\end{itemize}
If all the tuples have been tested and the algorithm did not stop to give a basis and an index of a subgroup, then it follows from Corollary \ref{cor:SubAlgoFI} that $C(\Psi)$ is not isomorphic to a finite-index subgroup of $C(\Phi)$.
\end{proof}

\begin{remark}
It is worth noticing that, because a cut-and-paste cannot increase the covolume of a peripheral collection of hyperplanes according to Lemma \ref{lem:Covolume}, the algorithm described in the proof of Theorem \ref{thm:AlgoFI} can be used to determine the minimal index of a finite-index subgroup of $C(\Psi)$ isomorphic to $C(\Phi)$, if such a subgroup exists. 
\end{remark}

\section{Concluding remarks and open problem}\label{section:Q}

\noindent
Theorem \ref{thm:MorphismReflection} shows that, given a morphism $\varphi : C(\Gamma_1) \to C(\Gamma_2)$ sending each generator of $C(\Gamma_1)$ to a reflection of $C(\Gamma_2)$, it is possible to pre-compose $\varphi$ with an alternating product of foldings and partial conjugations in order to obtain a peripheral embedding. Then, it is natural to ask whether such an alternating product may be replaced with a product of a folding with a product of partial conjugations. The next example shows that it is not the case. 
\begin{figure}
\begin{center}
\includegraphics[scale=0.38]{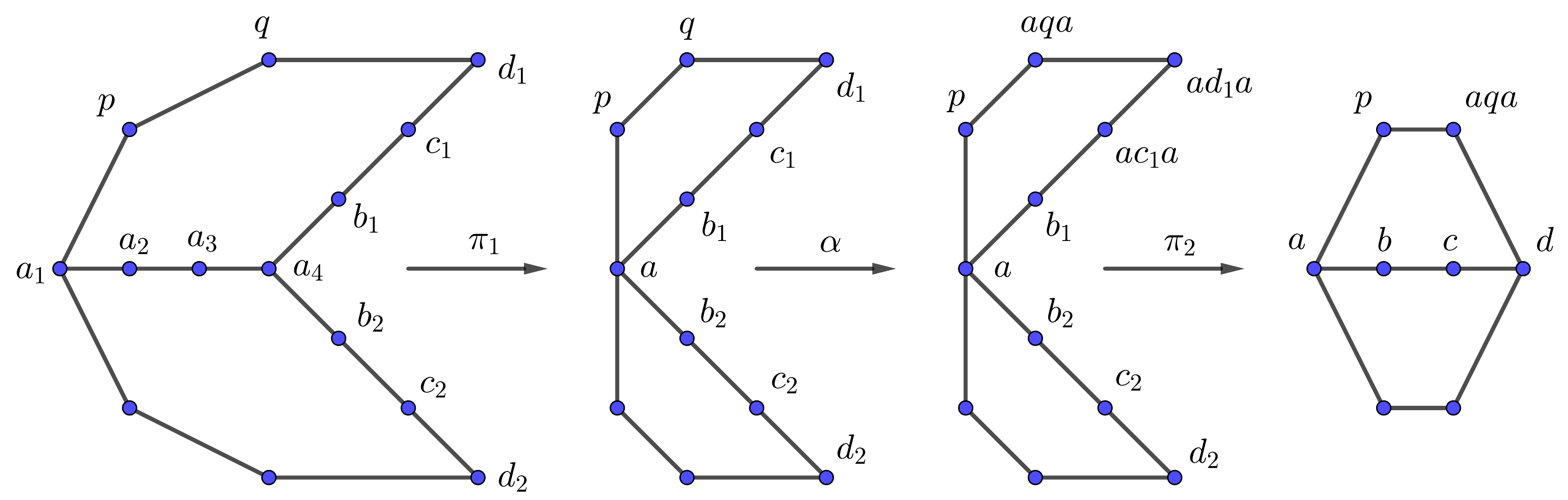}
\caption{The morphism $\pi_2 \circ \alpha \circ \pi_1$ is not the product of a folding with a conjugation.}
\label{altern}
\end{center}
\end{figure}

\begin{ex}\label{ex:alternating}
Figure \ref{altern} illustrates the alternating product $\pi_2 \circ \alpha \circ \pi_1$ of two foldings $\pi_1 : C(\Phi_1) \to C(\Phi_2)$, $\pi_2 : C(\Phi_2) \to C(\Phi_3)$ and a partial conjugation $\alpha \in \mathrm{Aut}(C(\Phi_2))$. Notice that $\Phi_1$ and $\Phi_2$ do not contain separating stars, so, if $\pi_2 \circ \alpha \circ \pi_1$ decomposes as the product of a folding with a product of particular conjugation, then it has to decompose as the product of a folding with a conjugation. But it is clearly not the case.
\end{ex}

\noindent
In Section \ref{section:peripheral}, we described cuts-and-pastes geometrically. However, applying a cut-and-paste to a peripheral embedding $\varphi : C(\Phi) \to C(\Psi)$ makes sense, and it would be interesting to have an algebraic description of this operation. Because it amounts to conjugating by elements of $C(\Psi)$ the images under $\varphi$ of the generators of $C(\Phi)$, it is natural ask if the cut-and-paste of $\varphi$ can be obtained from $\varphi$ by post-composing it with partial conjugations of $C(\Psi)$. The next example shows however that it is not the case.

\begin{ex}
Let $\Gamma$ be a square, and let $a,b,c,d$ denote its successive vertices. Notice that $\mathcal{J}_n= \{ (ab)^{-n} J_c, (ab)^nJ_d, (cd)^{-n}J_a, (cd)^n J_b\}$ is peripheral. (Indeed, $X(\Gamma)$ is the usual square tessellation of the Euclidean plane, and $\mathcal{J}_n$ is a cycle of four hyperplanes delimiting a region which contains $1$.) Let $\varphi_n : C(\Gamma) \hookrightarrow C(\Gamma)$ be the corresponding peripheral embedding. It is clear that the $\varphi_n$'s equivalent modulo cuts-and-pastes. However, the partial conjugations of $C(\Gamma)$ are all conjugations, and the $\varphi_n$'s are pairwise non-conjugate.
\end{ex}

\noindent
Regarding our description of morphisms between two-dimensional right-angled Coxeter groups, a natural question to ask is:

\begin{question}
For which finite (triangle-free) simplicial graph $\Gamma$ is the right-angled Coxeter group $C(\Gamma)$ cohopfian? 
\end{question}

\noindent
Recall that a group $G$ is \emph{cohopfian} if every injective morphism $G \hookrightarrow G$ is automatically surjective. Theorem \ref{thm:Algo} provides an algorithm to determine whether or not a given two-dimensional right-angled Coxeter group is cohopfian, but it does not describe the family of graphs leading to cohopfian right-angled Coxeter groups. Nevertheless, we expect that the techniques introduced in this paper will lead to a better understanding of the question. (For instance, based on the ideas used in the proof of Proposition \ref{prop:MainCut}, it is possible to prove that one-ended two-dimensional hyperbolic right-angled Coxeter groups are cohopfian, which is a consequence of a much more general criterion \cite[Th\'eor\`eme 1.0.7]{Moioli}.)

\medskip \noindent
Finally, the most natural (but also the most difficult) question is: what does happen for right-angled Coxeter groups of dimension at least three? In this case, reflection subgroups have to be replaced with \emph{generalised reflection subgroups}, where a \emph{generalised reflection} is a product of reflections along a collection of pairwise transverse hyperplanes (or equivalently, a product of pairwise commuting reflections). Unfortunately, the ping-pong used in Section \ref{section:pingpong} does not work anymore, and all our arguments fail. So the first step towards the understanding of morphisms between higher-dimensional right-angled Coxeter groups is to study the structure of generalised reflection subgroups. For instance, we do not know the answer to the following question:

\begin{question}
Is a generalised reflection subgroup a right-angled Coxeter group?
\end{question}

\addcontentsline{toc}{section}{References}

\bibliographystyle{alpha}
{\footnotesize\bibliography{MorphismsRACG}}

\end{document}